\documentclass{amsart}
\usepackage{graphicx}
\usepackage{amssymb}
\usepackage{amsfonts}
\swapnumbers
\newtheorem{theorem}{Theorem}[section]
\newtheorem{lemma}[theorem]{Lemma}

\newtheorem{proposition}[theorem]{Proposition}
\theoremstyle{definition}

\newtheorem{remark}[theorem]{Remark}

\numberwithin{equation}{section}
 \theoremstyle{plain}
 
 \numberwithin{equation}{section} 
 \numberwithin{figure}{section} 
 \theoremstyle{plain}
 \theoremstyle{remark}
 \newtheorem*{acknowledgement*}{Acknowledgement}

\newcommand{\cB}{{\mathcal B}}
\newcommand{\cC}{{\mathcal C}}
\newcommand{\cD}{{\mathcal D}}

\newcommand{\cF}{{\mathcal F}}
\newcommand{\cG}{{\mathcal G}}
\newcommand{\cH}{{\mathcal H}}
\newcommand{\cI}{{\mathcal I}}

\newcommand{\cP}{{\mathcal P}}

\newcommand{\cW}{{\mathcal W}}
\newcommand{\cX}{{\mathcal X}}

\newcommand{\fW}{{\mathfrak W}}

\newcommand{\vt}{{\vartheta}}
\newcommand{\Om}{{\Omega}}
\newcommand{\om}{{\omega}}
\newcommand{\ve}{{\varepsilon}}
\newcommand{\del}{{\delta}}

\newcommand{\gam}{{\gamma}}
\newcommand{\Gam}{{\Gamma}}

\newcommand{\vr}{{\varrho}}
\newcommand{\Sig}{{\Sigma}}
\newcommand{\sig}{{\sigma}}
\newcommand{\al}{{\alpha}}
\newcommand{\be}{{\beta}}
\newcommand{\ka}{{\kappa}}

\newcommand{\vp}{{\varpi}}
\newcommand{\up}{{\upsilon}}

\newcommand{\vs}{{\varsigma}}

\newcommand{\bbC}{{\mathbb C}}

\newcommand{\bbF}{{\mathbb F}}

\newcommand{\bbK}{{\mathbb K}}

\newcommand{\bbR}{{\mathbb R}}
\newcommand{\bbS}{{\mathbb S}}

\newcommand{\bbZ}{{\mathbb Z}}
\newcommand{\bbI}{{\mathbb I}}
\newcommand{\bbW}{{\mathbb W}}
\newcommand{\bbU}{{\mathbb U}}
\newcommand{\bbX}{{\mathbb X}}
\newcommand{\bbV}{{\mathbb V}}

\newcommand{\bfM}{{\bf M}}

\newcommand{\bfW}{{\bf W}}

\begin{document}
\title[]{Almost sure approximations and laws of iterated logarithm for signatures}%
 \vskip 0.1cm
 \author{ Yuri Kifer\\
\vskip 0.1cm
 Institute  of Mathematics\\
Hebrew University\\
Jerusalem, Israel}%
\address{
Institute of Mathematics, The Hebrew University, Jerusalem 91904, Israel}
\email{ kifer@math.huji.ac.il}%

\thanks{ }
\subjclass[2000]{Primary:  60F15 Secondary: 60L20, 37A50}%
\keywords{rough paths, signatures, strong approximations, Berry-Esseen estimates,
$\al$, $\phi$- and $\psi$-mixing, stationary process, shifts, dynamical systems.}%
\dedicatory{  }
 \date{\today}
\begin{abstract}\noindent
We obtain strong invariance principles for normalized multiple iterated sums and integrals
of the form $\bbS_N^{(\nu)}(t)=N^{-\nu/2}\sum_{0\leq k_1<...<k_\nu\leq Nt}\xi(k_1)\otimes\cdots\otimes\xi(k_\nu)$,
$t\in[0,T]$ and $\bbS_N^{(\nu)}(t)=N^{-\nu/2}\int_{0\leq s_1\leq...\leq s_\nu\leq Nt}\xi(s_1)\otimes\cdots\otimes\xi(s_\nu)ds_1\cdots ds_\nu$, where $\{\xi(k)\}_{-\infty<k<\infty}$ and $\{\xi(s)\}_{-\infty<s<\infty}$ are centered stationary vector processes with some weak
dependence properties. These imply also laws of iterated logarithm and an almost sure central limit theorem
for such objects. In the continuous time we work both under direct weak dependence
assumptions and also within the suspension setup which is more appropriate for applications in dynamical systems. Similar results under substantially
more restricted conditions were obtained in
\cite{FK} relying heavily on rough paths theory and notations while here we obtain these results in a more direct way which makes them accessible to
a wider readership. This is a companion paper of \cite{Ki23} and we consider a similar setup and rely on many result from there.
\end{abstract}
\maketitle
\markboth{Yu.Kifer}{Almost sure approximations for signatures}
\renewcommand{\theequation}{\arabic{section}.\arabic{equation}}
\pagenumbering{arabic}

\section{Introduction}\label{sec1}\setcounter{equation}{0}

Let $\{\xi(k)\}_{-\infty<k<\infty}$ and $\{\xi(t)\}_{-\infty<t<\infty}$ be discrete and continuous time $d$-dimensional
stationary processes such that for $s=0$ (and so for all $s$),
\begin{equation}\label{1.1}
E\xi(s)=0.
\end{equation}
The sequences of multiple iterated sums
\[
\Sig^{(\nu)}(v)=\sum_{0\leq k_1<...<k_\nu< [v]}\xi(k_1)\otimes\cdots\otimes\xi(k_\nu),
\]
in the discrete time, and of multiple iterated integrals
\[
\Sig^{(\nu)}(v)=\int_{0\leq u_1<...<u_\nu\leq v}\xi(u_1)\otimes\cdots\otimes\xi(u_\nu)du_1\cdots du_\nu,
\]
in the continuous time, were called signatures in recent papers related to the rough paths theory, data sciences and machine
learning (see, for instance, \cite{HL}, \cite{DR}, \cite{DET} and references there). Observe that for $\nu=1$ we have above usual sums and integrals.

In this paper we will study almost sure (a.s.) approximations for the normalized iterated sums and integrals $\bbS_N^{(\nu)}(t)=
N^{-\nu/2}\Sig^{(\nu)}(Nt),\, t\in[0,T]$. Under certain weak dependence conditions on the process $\xi$ it was shown
in \cite{FK} that there exists a Brownian motion with covariances $\cW$ such that the $\nu$-th term of the, so called,
Lyons' extension $\bbW_N^{(\nu)}$ constructed recursively (see (\ref{rec2}) in Section \ref{sec2})
with the rescaled Brownian motion $W_N(s)=N^{-1/2}\cW(Ns)$ and a certain drift term satisfy
\begin{equation}\label{1.2}
\sup_{0\leq t\leq T}|\bbS_N^{(\nu)}(t)-\bbW^{(\nu)}_N(t)|=O(N^{-\ve})\,\,\,\mbox{almost surely (a.s.)}
\end{equation}
for some $\ve>0$ which does not depend on $N$ and, in fact, this was proved in \cite{FK} for the $p$-variational and not just for the supremum norm. This is the strong (or a.s.) invariance principle for iterated sums and integrals and it was proved in \cite{FK} under boundedness and $\phi$-mixing conditions on the process
$\xi$ relying heavily on the results and notations of the rough paths theory. In this paper we will provide a more direct proof of such results under more general
moment and mixing (weak dependence) conditions which will make these results more accessible for a general probability readership. Relying on (\ref{1.2}) and the method from \cite{LP} we will derive also the, so called, almost sure central
limit theorem for iterated sums and integrals $\bbS_N^{(\nu)}$.

We will work under the setup of
our companion paper \cite{Ki23} where moment estimates for the $p$-variational distance between $\bbS_N^{(\nu)}$ and $\bbW^{(\nu)}_N$ were obtained
and will rely here on some of the results from there. Still, this paper can be read independently of \cite{Ki23} turning to the latter from time to time for
more details as we provide here all necessary setups and statements.
Of course, both the results from \cite{Ki23} and from here imply, in particular,
 the weak convergence of distributions of $\bbS_N^{(\nu)}$ to the distribution of $\bbW_N^{(\nu)}$ (which does not depend on $N$). The latter does not seem to be stated before \cite{Ki23} in the full generality though for $\nu=1$ and $\nu=2$ this was proved before for certain classes of processes ( see, for instance,
\cite{CFKMZ} and references there).

In the continuous time case we consider two setups. The first one is standard in probability when we impose
mixing and approximation conditions with respect to a family of $\sig$-algebras indexed by two continuous
time parameters on the same
probability space on which our continuous time process $\xi$ is defined. Since this setup does not have many
applications to processes generated by continuous time dynamical systems (flows) we consider also another setup,
called suspension, when mixing and approximation conditions are imposed on a discrete time process defined on
on a base probability space and the continuous time process $\xi$ moves deterministically for the time determined by certain ceiling function and then jumps to the base according to the above discrete time process. This construction
is adapted to applications for certain important classes of dynamical systems, i.e. when $\xi(n)=g\circ F^n$ or
 $\xi(t)=g\circ F^t$ where $F$ is a measure preserving transformation or a continuous time measure preserving flow and
 $g$ is a (vector) function. Unlike \cite{FK} and several other related papers we work under quite general dependence (mixing) and
 moment conditions and not under specific $\alpha$- or $\phi$-mixing assumptions.

 The above a.s. approximations results for the normalized iterated sums and integrals enable us to obtain laws of iterated logarithm
 for them employing corresponding results on such laws for iterated stochastic integrals from \cite{Bal}. Similar results under more
 restricted conditions were obtained in \cite{FK} but the arguments there are rather sketchy and they rely heavily on the rough paths theory,
 so our more direct approach should benefit a general probability reader.
 
 We observe that the iterated sums and integrals $\Sig^{(\nu)}$ defined above have a visual resemblance with $U$-statistics having special (degenerate) 
 product type kernels. The kernels of $U$-statistics are usually supposed to be symmetric (see, for instance  \cite{DK}, \cite{DP}, \cite{KY} and references
 there) while in our case for $x,y\in\bbR^d$ with $d>1$, in general, $h_{ij}(x,y)=x_iy_j\ne y_ix_j=h_{ij}(x,y)$. Only when $x_1,x_2,...,x_\nu$ are one dimensional
 $h(x_1,...,x_\nu)=x_1x_2\cdots x_\nu$ becomes a symmetric degenerate kernel of an $U$-statistics. Furthermore, $U$-statistics are usually the one dimensional objects and
 so are limit theorems for them while we obtain limit theorems for the whole tensor products defined above viewed as $d^\nu$-dimensional objects and such results
 cannot be obtained relying on limit theorems for each of their components separately.
Most of the papers on $U$-statistics having degenerated kernels
 deal with weak limit theorems for them (see \cite{BV} and references there) with representations of limit processes somewhat different from ours. For $U$-statistics
 with bivariate degenerate kernels almost sure approximation type limit theorems were obtained in \cite{DDP} (with independent random variables) and in \cite{KY}
 (with weakly dependent stationary sequences). Our motivation and methods are different from papers on $U$-statistics and we obtain estimates in stronger variational 
 norms. In order to extend our estimates from sums and integrals to multiple iterated sums and integrals we rely on Chen relations (see Theorem \ref{hoelder} and
  Proposition \ref{hoelder2}) which hold true for iterated sums and integrals but not for general $U$-statistics. Consequently, the limit theorems we obtain in this 
  paper are of different types than in the works of $U$-statistics.

 The structure of this paper is the following. In the next section we provide necessary definitions and give
   precise statements of our results. Section \ref{sec3} is devoted to necessary estimates both of general nature
   and more specific to our problem as well as characteristic functions approximations needed
    for the strong approximation theorem which concludes that section. In Section \ref{sec4} we obtain
    main estimates for iterated sums with respect to the variational norms while Section \ref{sec5} is devoted to the
    "straightforward" and suspension continuous time setups. In Section \ref{sec6} we extend the results to multiple iterated sums and integrals relying on results from \cite{Ki23}. In Section \ref{sec7} we derive laws of iterated logarithm
     for multiple iterated sums and integrals and in Section \ref{sec8} we prove our almost sure central limit theorem
     for these objects.

\section{Preliminaries and main results}\label{sec2}\setcounter{equation}{0}
\subsection{Discrete time case}\label{subsec2.1}
We start with the discrete time setup which consists of a complete probability space
$(\Om,\cF,P)$, a stationary sequence of $d$-dimensional centered random vectors $\xi(n)=(\xi_1(n),...,\xi_d(n))$,
 $-\infty<n<\infty$, and a two-parameter family of countably generated $\sig$-algebras
$\cF_{m,n}\subset\cF,\,-\infty\leq m\leq n\leq\infty$ such that
$\cF_{mn}\subset\cF_{m'n'}\subset\cF$ if $m'\leq m\leq n
\leq n'$, where $\cF_{m\infty}=\cup_{n:\, n\geq m}\cF_{mn}$ and $\cF_{-\infty n}=\cup_{m:\, m\leq n}\cF_{mn}$.
It is often convenient to measure the dependence between two sub
$\sig$-algebras $\cG,\cH\subset\cF$ via the quantities
\begin{equation}\label{2.1}
\varpi_{b,a}(\cG,\cH)=\sup\{\| E(g|\cG)-Eg\|_a:\, g\,\,\mbox{is}\,\,
\cH-\mbox{measurable and}\,\,\| g\|_b\leq 1\},
\end{equation}
where the supremum is taken over real-valued functions and $\|\cdot\|_c$ is the
$L^c(\Om,\cF,P)$-norm. Then more familiar $\al,\rho,\phi$ and $\psi$-mixing
(dependence) coefficients can be expressed via the formulas (see \cite{Bra},
Ch. 4 ),
\begin{eqnarray*}
&\al(\cG,\cH)=\frac 14\varpi_{\infty,1}(\cG,\cH),\,\,\rho(\cG,\cH)=\varpi_{2,2}
(\cG,\cH)\\
&\phi(\cG,\cH)=\frac 12\varpi_{\infty,\infty}(\cG,\cH)\,\,\mbox{and}\,\,
\psi(\cG,\cH)=\varpi_{1,\infty}(\cG,\cH).
\end{eqnarray*}
We set also
\begin{equation}\label{2.2}
\varpi_{b,a}(n)=\sup_{k\geq 0}\varpi_{b,a}(\cF_{-\infty,k},\cF_{k+n,\infty})
\end{equation}
and accordingly
\[
\al(n)=\frac{1}{4}\varpi_{\infty,1}(n),\,\rho(n)=\varpi_{2,2}(n),\,
\phi(n)=\frac 12\varpi_{\infty,\infty}(n),\, \psi(n)=\varpi_{1,\infty}(n).
\]
Our setup includes also the  approximation rate
\begin{equation}\label{2.3}
\beta (a,l)=\sup_{k\geq 0}\|\xi(k)-E(\xi(k)|\cF_{k-l,k+l})\|_a,\,\, a\geq 1.
\end{equation}
We will assume that for some $ 1\leq L\leq\infty$, $M\geq 1$ large enough (whose value will be specified in Sections \ref{sec4}--\ref{sec6})
 and $K=\max(2L,4M)$,
\begin{equation}\label{2.4}
E|\xi(0)|^{K}<\infty,\,\,\,\mbox{and}\,\,\,\sum_{l=1}^\infty l(\sqrt{\sup_{m\geq l}\be(K,m)}+\varpi_{L,4M}(l))<\infty.
\end{equation}

In order to formulate our results we have to introduce also the "increments" of multiple iterated sums under consideration
\begin{equation}\label{2.5}
\Sig^{(\nu)}(u,v)=\sum_{[u]\leq k_1<...<k_\nu< [v]}\xi(k_1)\otimes\cdots\otimes\xi(k_\nu),\, 0\leq u\leq v
\end{equation}
which means that $\Sig^{(\nu)}(u,v)=\{\Sig^{i_1,...,i_\nu}(u,v),\, 1\leq i_1,...,i_\nu\leq d\}$ where
in the coordinate-wise form
\[
\Sig^{i_1,...,i_\nu}(u,v)=\sum_{[u]\leq k_1<...<k_\nu< [v]}\xi_{i_1}(k_1)\cdots\xi_{i_\nu}(k_\nu).
\]
We set also for any $0\leq s\leq t\leq T$,
\[
\bbS_N^{(\nu)}(s,t)=N^{-\nu/2}\Sig^{(\nu)}(sN,tN),\,\,\bbS_N(s,t)=\bbS_N^{(2)}(s,t)\quad\mbox{and}\quad S_N(s,t)=\bbS_N^{(1)}(s,t).
\]
Recall, that when $u=s=0$ we just write
\[
\Sig^{(\nu)}(v)=\Sig^{(\nu)}(0,v),\,\bbS_N^{(\nu)}(t)=\bbS_N^{(\nu)}(0,t),\,\bbS_N(t)=\bbS_N(0,t)\,\,\mbox{and}\,\,
S_N(t)=S_N(0,t).
\]

Next, introduce also the covariance matrix $\vs=(\vs_{ij})$ defined by
\begin{equation}\label{2.6}
\vs_{ij}=\lim_{k\to\infty}\frac 1k\sum_{m=0}^k\sum_{n=0}^k\vs_{ij}(n-m),\,\,
\mbox{where}\,\, \vs_{ij}(n-m)=E(\xi_i(m)\xi_j(n))
\end{equation}
taking into account that the limit here exists under conditions of our theorem below (see Section \ref{sec3}).
Let $\cW=(\cW^1,\cW^2,...,\cW^d)$ be a $d$-dimensional Brownian motion with the covariance matrix $\vs$ (at the time 1) and introduce
the rescaled Brownian motion $\bbW^{(1)}_N(t)=W_N(t)=N^{-1/2}\cW(Nt),\,  t\in[0,T],\, N\geq 1$. We set also
$\bbW_N^{(1)}(s,t)=W_N(s,t)=W_N(t)-W_N(s),\, t\geq s\geq 0$. Next, we introduce
\begin{equation}\label{rec1}
\bbW_N^{(2)}(s,t)=\bbW_N(s,t)=\int_s^tW_N(s,v)\otimes dW_N(v)+(t-s)\Gam
\end{equation}
which can be written in the coordinate-wise form as
\[
\bbW_N^{ij}(s,t)=\int_s^tW_N^i(s,v)dW^j_N(v)+(t-s)\Gam^{ij}\,\,\,\mbox{where}\,\,\,\Gam^{ij}=\sum_{l=1}^\infty
E(\xi_i(0)\xi_j(l))
\]
and the latter series converges absolutely as we will see in Section \ref{sec3}. Again, we set $\bbW_N^{(2)}(t)=
\bbW_N(t)=\bbW_N(0,t)$. For $n>2$ we define recursively,
\begin{equation}\label{rec2}
\bbW_N^{(n)}(s,t)=\int_s^t\bbW_N^{(n-1)}(s,v)\otimes dW_N(v)+\int_s^t\bbW_N^{(n-2)}(s,v)\otimes\Gam dv.
\end{equation}
Both here and above the stochastic integrals are understood in the It\^ o sense. Coordinate-wise this relation
can be written in the form
\begin{eqnarray*}
&\bbW_N^{i_1,...,i_n}(s,t)=\int_s^t\bbW_N^{i_1,...,i_{n-1}}(s,v) dW^{i_n}_N(v)\\
&+\int_s^t\bbW_N^{i_n,...,i_{n-2}}(s,v)\Gam^{i_{n-1}i_n} dv.
\end{eqnarray*}
As before, we write also $\bbW_N^{(n)}(t)=\bbW_N^{(n)}(0,t)$.

In order to formulate our first result we recall the definition of $p$-variation norms. For any path $\gam(t),\, t\in[0,T]$ in a Euclidean space having left and right limits and $p\geq 1$ the $p$-variation norm of
$\gam$ on  an interval $[U,V],\, U<V$ is given by
 \begin{equation}\label{2.7}
 \|\gam\|_{p,[U,V]}=\big(\sup_\cP\sum_{[s,t]\in\cP}|\gam(s,t)|^p\big)^{1/p}
 \end{equation}
 where the supremum is taken over all partitions $\cP=\{ U=t_0<t_1<...<t_n=V\}$ of $[U,V]$ and the sum is taken over
 the corresponding subintervals $[t_i,t_{i+1}],\, i=0,1,...,n-1$ of the partition while $\gam(s,t)$ is taken according
 to the definitions above depending on the process under consideration. In what follows always $2<p<3$ but 
 the notation (\ref{2.7}) will be used for any positive $q$ in place of $p$ there. We will prove


\begin{theorem}\label{thm2.1}
Let (\ref{2.4}) holds true with integers $L\geq 1$ and a large enough $M$. Then the
stationary sequence of random vectors $\xi(n),\,-\infty<n<\infty$ can be redefined preserving its joint distribution on
a sufficiently rich probability space which contains  also a $d$-dimensional Brownian motion $\cW$ with the covariance
matrix $\vs$ (at the time 1) so that for any integer $N\geq 1$ the processes $\bbW_N^{(\nu)},\,
\nu=1,2,...$ constructed as above with the rescaled Brownian motion $W_N(t)=N^{-1/2}\cW(Nt),\,  t\in[0,T]$
satisfy,
 \begin{equation}\label{2.8}
 \| \bbS^{(\nu)}_N-\bbW^{(\nu)}_N\|_{p/\nu,[0,T]}=O(N^{-\ve_\nu})\,\,\,\mbox{a.s.},\,\,\,\nu=1,2,...,\nu(M)
 \end{equation}
 where the constants $\ve_\nu>0$ do not depend on $N$ and $\nu(M)$ depends on $M$ but not on $N$. Moreover, $\ve_\nu$ can
 be taken of order $(\nu !)^{-2}q^\nu$ where $q=\frac {p^2}{12}(\frac p2-1)$ while $\nu(M)$ can be chosen
 so that $(\nu(M) !)^{2}q^{-\nu}$ has the order of $M$. In particular, $\nu(M)\to\infty$ as $M\to\infty$.
  If (\ref{2.4}) holds true for any $M\geq 1$ then (\ref{2.8}) is satisfied for any $\nu\geq 1$.
 \end{theorem}
 
 We stress that a sufficiently rich probability space has here a quite precise meaning that one can define on it a sequence of 
independent uniformly didtributed on $[0.1]$ random variables which are independentof the sequence $\xi(n),\, n\in\bbZ$ 
(see Theorem \ref{thm3.8} below.

 In order to understand our assumptions observe that $\varpi_{q,p}$
is clearly non-increasing in $b$ and non-decreasing in $a$. Hence, for any pair $a,b\geq 1$,
\[
\varpi_{b,a}(n)\leq\psi(n).
\]
Furthermore, by the real version of the Riesz--Thorin interpolation
theorem or the Riesz convexity theorem (see \cite{Ga}, Section 9.3
and \cite{DS}, Section VI.10.11) whenever $\theta\in[0,1],\, 1\leq
a_0,a_1,b_0,b_1\leq\infty$ and
\[
\frac 1a=\frac {1-\theta}{a_0}+\frac \theta{a_1},\,\,\frac 1b=\frac
{1-\theta}{b_0}+\frac \theta{b_1}
\]
then
\begin{equation}\label{2.9}
\varpi_{b,a}(n)\le 2(\varpi_{b_0,a_0}(n))^{1-\theta}
(\varpi_{b_1,a_1}(n))^\theta.
\end{equation}
In particular,  using the obvious bound $\varpi_{b_1,a_1}(n)\leq 2$
valid for any $b_1\geq a_1$ we obtain from (\ref{2.9}) for pairs
$(\infty,1)$, $(2,2)$ and $(\infty,\infty)$ that for all $b\geq a\geq 1$,
\begin{eqnarray}\label{2.10}
&\varpi_{b,a}(n)\le 4(2\alpha(n))^{\frac{1}{a}-\frac{1}{b}},\,
\varpi_{b,a}(n)\le 2^{1+\frac 1a-\frac 1b}(\rho(n))^{1-\frac 1a+\frac 1b}\\
&\mbox{and}\,\,\varpi_{b,a}(n)\le 2^{1+\frac 1a}(\phi(n))^{1-\frac 1a}.
\nonumber\end{eqnarray}
We observe also that by the H\" older inequality for $b\geq a\geq 1$
and $\alpha\in(0,a/b)$,
\begin{equation}\label{2.11}
\beta(b,l)\le 2^{1-\alpha}  [\beta(a,l)]^\alpha \|\xi(0)\|^{1-\al}_{\frac{ab(1-\al)}
{a-b\al}}.
\end{equation}
 Thus, we can formulate assumption (\ref{2.4}) in terms of more familiar $\alpha,\,\rho,\,\phi,$
and $\psi$--mixing coefficients and with various moment conditions. It follows also from (\ref{2.9})
 that if $\varpi_{b,a}(n)\to 0$ as $n\to\infty$ for some $b>a\geq 1$ then
\begin{equation}\label{2.12}
\varpi_{b,a}(n)\to 0\,\,\mbox{as}\,\, n\to\infty\,\,\mbox{for all}\,\, b> a\geq 1,
\end{equation}
and so (\ref{2.12}) follows from (\ref{2.4}).

Observe that the estimate (\ref{2.8}) in the $p$-variation norm is stronger than an estimate just in the
supremum norm. In order to prove Theorem \ref{thm2.1} we will first derive directly (\ref{2.8}) for $\nu=1$
and $\nu=2$ relying, in particular, on the strong approximation theorem. Since the latter result did not seem
to appear before under our moment and mixing conditions we will provide the details which cannot be found in the
earlier literature. Observe that our mixing assumptions in (\ref{2.4}) together with the inequality (\ref{2.10})
 allow to obtain the strong approximation theorem under more general conditions than the ones appeared before.
Having (\ref{2.8}) for $\nu=1,2$ we will employ the results from \cite{Ki23} which are based on, so called, Chen's
relations in order to extend (\ref{2.8}) directly from $\nu=1,2$ to $\nu\geq 3$.

  Important classes of processes satisfying our conditions come from
dynamical systems. Let $F$ be a $C^2$ Axiom A diffeomorphism (in
particular, Anosov) in a neighborhood $\Om$ of an attractor or let $F$ be
an expanding $C^2$ endomorphism of a compact Riemannian manifold $\Om$ (see
\cite{Bow}), $g$ be either a H\" older continuous vector function or a
vector function which is constant on elements of a Markov partition and let $\xi(n)=
\xi(n,\om)=g(F^n\om)$. Here the probability space is $(\Om,\cB,P)$ where $P$ is a Gibbs
 invariant measure corresponding to some H\"older continuous function and $\cB$ is the Borel $\sig$-field.
  Let $\zeta$ be a finite Markov partition for $F$ then we can take $\cF_{kl}$
 to be the finite $\sig$-algebra generated by the partition $\cap_{i=k}^lF^i\zeta$.
 In fact, we can take here not only H\" older continuous $g$'s but also indicators
of sets from $\cF_{kl}$. The conditions of Theorems \ref{thm2.1} allow all such functions
since the dependence of H\" older continuous functions on $m$-tails, i.e. on events measurable
with respect to $\cF_{-\infty,-m}$ or $\cF_{m,\infty}$, decays exponentially fast in $m$ and
the condition (\ref{2.4}) is much weaker than that. A related class of dynamical systems
corresponds to $F$ being a topologically mixing subshift of finite type which means that $F$
is the left shift on a subspace $\Om$ of the space of one (or two) sided
sequences $\om=(\om_i,\, i\geq 0), \om_i=1,...,l_0$ such that $\om\in\Om$
if $\pi_{\om_i\om_{i+1}}=1$ for all $i\geq 0$ where $\Pi=(\pi_{ij})$
is an $l_0\times l_0$ matrix with $0$ and $1$ entries and such that $\Pi^n$
for some $n$ is a matrix with positive entries. Again, we have to take in this
case $g$ to be a H\" older continuous bounded function on the sequence space above,
 $P$ to be a Gibbs invariant measure corresponding to some H\" older continuous function and to define
$\cF_{kl}$ as the finite $\sig$-algebra generated by cylinder sets
with fixed coordinates having numbers from $k$ to $l$. The
exponentially fast $\psi$-mixing is well known in the above cases (see \cite{Bow}) and this property
is much stronger than what we assume in (\ref{2.4}). Among other
dynamical systems with exponentially fast $\psi$-mixing we can mention also the Gauss map
$Fx=\{1/x\}$ (where $\{\cdot\}$ denotes the fractional part) of the
unit interval with respect to the Gauss measure $G$ and more general transformations generated
by $f$-expansions (see \cite{Hei}). Gibbs-Markov maps which are known to be exponentially fast
$\phi$-mixing (see, for instance, \cite{MN}) can be also taken as $F$ in Theorem \ref{thm2.1}
with $\xi(n)=g\circ F^n$ as above.

\subsection{Straightforward continuous time setup}\label{subsec2.2}

Our direct continuous time setup consists of a Lebesgue integrable $d$-dimensional stationary process
$\xi(t),\, t\geq 0$ on a probability space $(\Om,\cF,P)$ satisfying (\ref{1.1}) and of a family of
$\sig$-algebras $\cF_{st}\subset\cF,\,-\infty\leq s\leq t\leq\infty$ such
that $\cF_{st}\subset\cF_{s't'}$ if $s'\leq s$ and $t'\geq t$. For all $t\geq 0$ we set
\begin{equation}\label{2.14}
\varpi_{b,a}(t)=\sup_{s\geq 0}\varpi_{b,a}(\cF_{-\infty,s},\cF_{s+t,\infty})
\end{equation}
and
\begin{equation}\label{2.15}
\beta (a,t)=\sup_{s\geq 0}\|\xi(s)-E(\xi(s)|\cF_{s-t,s+t})\|_a.
\end{equation}
where $\varpi_{b,a}(\cG,\cH)$ is defined by (\ref{2.1}).  We continue to
impose the assumption (\ref{2.4}) on the  decay rates  of
$\varpi_{b,a}(t)$ and $\beta (a,t)$. Although they only involve integer
values of $t$, it will  suffice since these are non-increasing functions of $t$.

Next, we introduce the covariance matrix $\vs=(\vs_{ij})$ defined by
\begin{equation}\label{2.16}
\vs_{ij}=\lim_{t\to\infty}\frac 1t\int_{0}^t\int_{0}^t\vs_{ij}(u-v)dudv,\,\,
\mbox{where}\,\, \vs_{ij}(u-v)=E(\xi_i(u)\xi_j(v))
\end{equation}
and the limit here exists under our conditions in the same way as in the discrete time setup.
In order to formulate our results we define the "increments" of multiple iterated integrals
\[
\Sig^{(\nu)}(u,v)=\int_{u\leq u_1\leq...\leq u_\nu\leq v}\xi(u_1)\otimes\cdots\otimes\xi(u_\nu)du_1\cdots du_\nu,\, u\leq v
\]
which coordinate-wise have the form
\[
\Sig^{(\nu)}(u,v)=\{\Sig^{i_1,...,i_\nu}(u,v),\, 1\leq i_1,...,i_\nu\leq d\}
\]
with
\[
\Sig^{i_1,...,i_\nu}(u,v)=\int_{u\leq u_1\leq...\leq u_\nu\leq v}\xi_{i_1}(u_1)\xi_{i_2}(u_2)\cdots\xi_{i_\nu}(u_\nu)du_1\cdots du_\nu
\]
where we use the same letter $\Sig$ as in the discrete time case which should not lead to a confusion. All integrals here and below are supposed to exist.
As in the discrete time case we set also
\[
\bbS_N^{(\nu)}(s,t)=N^{-\nu/2}\Sig^{(\nu)}(sN,tN),\,\,\bbS_N^{i_1,...,i_\nu}(s,t)=
N^{-\nu/2}\Sig^{i_1,...,i_\nu}(sN,tN),
\]
$\bbS_N(s,t)=\bbS_N^{(2)}(s,t)$ and $S_N(s,t)=\bbS_N^{(1)}(s,t)$.   When $u=s=0$ we will write
\[
\Sig^{(\nu)}(v)=\Sig^{(\nu)}(0,v)\quad\mbox{and}\quad \bbS_N^{(\nu)}(t)=\bbS_N^{(\nu)}(0,t).
\]
Next, we introduce the matrix
\[
\Gam=(\Gam^{ij}),\,\,\Gam^{ij}=\int_1^\infty du\int_0^1E(\xi_i(v)\xi_j(u))dv+\int_0^1du\int_0^uE(\xi_i(v)\xi_j(u))dv.
\]
Then we have
\begin{theorem}\label{thm2.2}
Let (\ref{2.4}) holds true with integers $L\geq 1$ and a large enough $M$ where $\varpi$
and $\be$ are given by (\ref{2.14}) and (\ref{2.15}). Then the vector stationary process
  $\xi(t),\,-\infty<t<\infty$ can be redefined preserving its joint distribution on a sufficiently rich probability space which
   contains also a $d$-dimensional Brownian motion $\cW$ with the covariance matrix $\vs$ (at the time 1) so that for
  $\bbW_N^{(\nu)}$, constructed as in Theorem \ref{thm2.1} with  the rescaled Brownian motion
 $W_N(t)=N^{-1/2}\cW(Nt),\,  t\in[0,T]$ and the matrix $\Gam$ introduced above, and for any $N\geq 1$ the estimate (\ref{2.8})
 remains true for $\bbS_N^{(\nu)}$ with $\nu=1,...,\nu(M)$ defined above in the present continuous time setup and, again,
 $\nu(M)\to\infty$ as $M\to\infty$.
 \end{theorem}

\subsection{Continuous time suspension setup}\label{subsec2.3}

Here we start with a complete probability space $(\Om,\cF,P)$, a
$P$-preserving invertible transformation $\vt:\,\Om\to\Om$ and
a two parameter family of countably generated $\sig$-algebras
$\cF_{m,n}\subset\cF,\,-\infty\leq m\leq n\leq\infty$ such that
$\cF_{mn}\subset\cF_{m'n'}\subset\cF$ if $m'\leq m\leq n
\leq n'$ where $\cF_{m\infty}=\cup_{n:\, n\geq m}\cF_{mn}$ and
$\cF_{-\infty n}=\cup_{m:\, m\leq n}\cF_{mn}$. The setup includes
also a (roof or ceiling) function $\tau:\,\Om\to (0,\infty)$ such that
for some $\hat L>0$,
\begin{equation}\label{2.17}
\hat L^{-1}\leq\tau\leq\hat L.
\end{equation}
Next, we consider the probability space $(\hat\Om,\hat\cF,\hat P)$ such that $\hat\Om=\{\hat\om=
(\om,t):\,\om\in\Om,\, 0\leq t\leq\tau(\om)\},\, (\om,\tau(\om))=(\vt\om,0)\}$, $\hat\cF$ is the
restriction to $\hat\Om$ of $\cF\times\cB_{[0,\hat L]}$, where $\cB_{[0,\hat L]}$ is the Borel
$\sig$-algebra on $[0,\hat L]$ completed by the Lebesgue zero sets, and for any $\Gam\in\hat\cF$,
\[
\hat P(\Gam)=\bar\tau^{-1}\int\bbI_\Gam(\om,t)dtdP(\om)\,\,\mbox{where}\,\,\bar\tau=\int\tau dP=E\tau
\]
and $E$ denotes the expectation on the space $(\Om,\cF,P)$.

Finally, we introduce a Lebesgue integrable vector valued stochastic process $\xi(t)=\xi(t,(\om,s))$, $-\infty<t<\infty,\, 0\leq
s\leq\tau(\om)$ on $\hat\Om$ satisfying
\begin{eqnarray*}
&\int\xi(t)d\hat P=0,\,\,\xi(t,(\om,s))=\xi(t+s,(\om,0))=\xi(0,(\om,t+s))\,\,\mbox{if}\,\, 0\leq t+s<\tau(\om)\\
&\mbox{and}\,\,
\xi(t,(\om,s))=\xi(0,(\vt^k\om,u))\,\,\mbox{if}\,\, t+s=u+\sum_{j=0}^{k-1}\tau(\vt^j\om)\,\,\mbox{and}\,\,
0\leq u<\tau(\vt^k\om).
\end{eqnarray*}
This construction is called in dynamical systems a suspension and it is a standard fact that $\xi$ is a
stationary process on the probability space $(\hat\Om,\hat\cF,\hat P)$ and in what follows we will write
also $\xi(t,\om)$ for $\xi(t,(\om,0))$. In this setup we define
\[
\Sig^{(\nu)}(u,v)=\int_{u\bar\tau\leq u_1\leq...\leq u_\nu\leq v\bar\tau}\xi(u_1)\otimes\cdots\otimes\xi(u_\nu)du_1\cdots du_\nu,\, u\leq v,
\]
\[
\Sig^{i_1,...,i_\nu}(u,v)=\int_{u\bar\tau\leq u_1\leq...\leq u_\nu\leq v\bar\tau}\xi_{i_1}(u_1)\xi_{i_2}(u_2)\cdots\xi_{i_\nu}(u_\nu)du_1\cdots du_\nu,
\]
all integrals here and below are supposed to exist,
\[
\bbS_N^{(\nu)}(s,t)=N^{-\nu/2}\Sig^{(\nu)}(sN,tN),\,\,\,\bbS_N^{i_1,...,i_\nu}(s,t)=N^{-\nu/2}\Sig^{i_1,...,i_\nu}(sN,tN)
\]
and, again, $\bbS_N(s,t)=\bbS_N^{(2)}(s,t)$, $S_N(s,t)=\bbS_N^{(1)}(s,t)$,
\[
S_N(s,t)=\bbS_N^{(1)}(s,t),\,\,\Sig^{(\nu)}(v)=\Sig^{(\nu)}(0,v)\quad\mbox{and}\quad \bbS_N^{(\nu)}(t)=\bbS_N^{(\nu)}(0,t).
\]

Set $\eta(\om)=\int_0^{\tau(\om)}\xi(s,\om)ds$, $\eta(m)=\eta(m,\om)=\eta\circ\vt^m(\om)$ and
\begin{equation}\label{2.18}
\be(a,l)=\sup_m\max\big(\|\tau\circ\vt^m-E(\tau\circ\vt^m|\cF_{m-l,m+l})\|_a,\,\|\eta(m)-E(\eta(m)|\cF_{m-l,m+l})\|_a\big).
\end{equation}
We define $\varpi_{b,a}(n)$ by (\ref{2.2}) with respect to the $\sig$-algebras $\cF_{kl}$ appearing here.
Observe also that $\eta(k)=\eta\circ\vt^k$ is a stationary sequence of random vectors and we introduce also the covariance matrix
 \begin{equation}\label{2.19}
\vs_{ij}=\lim_{n\to\infty}\frac 1n\sum_{k,l=0}^nE(\eta_i(k)\eta_j(l))
\end{equation}
where the limit exists under our conditions in the same way as in (\ref{2.6}). We introduce also the matrix
\[
\Gam=(\Gam^{ij}),\,\,\Gam^{ij}=\sum_{l=1}^\infty E(\eta_i(0)\eta_j(l))+E\int_0^{\tau(\om)}\xi_j(s,\om)ds\int_0^s\xi_i(u,\om)du.
\]
The following is our limit theorem in the present setup.

\begin{theorem}\label{thm2.3}
Assume that $E\eta=E\int_0^\tau\xi(s)ds=0$ and $E\int_0^\tau |\xi(s)|^{ K}ds<\infty$. The latter replaces the moment
condition on $\xi(0)$ in (\ref{2.4}) while other conditions there are supposed to hold true with integers $L\geq 1$ and
a large enough $M$ where $\varpi$ is given by (\ref{2.2}) for $\sig$-algebras $\cF_{mn}$
appearing in this subsection and $\be$ is defined by (\ref{2.18}). Then the process $\xi(t),\,-\infty<t<\infty$ can be
redefined preserving its joint distribution on a sufficiently rich probability space which contains also a $d$-dimensional Brownian
motion $\cW$ with the covariance matrix $\vs$ (at the time 1) so that for $\bbW_N^{(\nu)}$, constructed
as in Theorem \ref{thm2.1} with the rescaled Brownian motion $W_N(t)=N^{-1/2}\cW(Nt),\,  t\in[0,T]$ and the matrix $\Gam$ introduced above, and for any $N\geq 1$
 the estimate (\ref{2.8}) remains true for $\bbS_N^{(\nu)}$ with $\nu=1,...,\nu(M)$ defined above.
 \end{theorem}

This theorem extends applicability of our results to hyperbolic flows (see \cite{BR}) which is an important class of continuous time dynamical systems.

\begin{remark}\label{rem2.3+}
By the definition $W_N(t)=N^{-1/2}\cW(Nt)$ and changing time in the stochastic integral
\[
\bbW_N(t)=\int_0^tW_N(v)\otimes dW_N(v)+t\Gam=N^{-1}(\int_0^{Nt}\cW(u)\otimes d\cW(u)+Nt\Gam)=N^{-1}\bbW_1(Nt).
\]
Continuing by induction we have
\begin{eqnarray*}
&\bbW_N^{(n)}(t)=\int_0^t\bbW_N^{(n-1)}(v)\otimes dW_N(v)+\int_0^t\bbW_N^{(n-2)}(v)\otimes\Gam dv\\
&=N^{-n/2}(\int_0^{Nt}\bbW_1^{(n-1)}(u)\otimes d\cW(u)+\int_0^{Nt}\bbW_1^{(n-2)}(u)\otimes\Gam du)=N^{-n/2}\bbW_1^{(n)}(Nt).
\end{eqnarray*}
Hence, we can replace $\bbW_N^{(\nu)}$ in (\ref{2.8}) (as well as in Theorems \ref{thm2.2} and \ref{thm2.3}) by 
$\hat\bbW_N^{(\nu)}$ where $\hat\bbW_N^{(\nu)}(t)=N^{-n/2}\bbW_1^{(\nu)}(Nt)$ which could be a helpful clarification
since $\bbW_1^{(\nu)}$ does not depend on $N$ and the dependence on $N$ appears here only in normalization and time
stretch coefficients.
\end{remark}

\subsection{Law of iterated logarithm}\label{subsec2.4}

For any $0\leq s\leq t<\infty$ and $\nu=1,2,...$ define inductively $\fW^{(1)}(s,t)=\cW(s,t)$,
$\fW_\Gam^{(2)}(s,t)=\int_s^t\cW(s,v)\otimes d\cW(v)+(t-s)\Gam$
and for $\nu=3,4,...$,
\[
\fW_\Gam^{(\nu)}(s,t)=\int_s^t\fW_\Gam^{(\nu-1)}(s,v)\otimes d\cW(v)+\int_s^t\fW_\Gam^{(\nu-2)}(s,v)\otimes\Gam dv
\]
where $\cW$ is the Brownian motion with a covariance matrix $\vs$ appearing in Theorems \ref{thm2.1}--\ref{thm2.3}.
Written coordinate-wise this recursive relation has the form
\[
\fW_\Gam^{i_1,...,i_\nu}(s,t)=\int_s^t\fW_\Gam^{i_1,...,i_{\nu-1}}(s,v)d\cW^{i_\nu}(v)+\int_s^t\fW_\Gam^{i_1,...,i_{\nu-2}}(s,v)\Gam^{i_{\nu-1},\i_\nu}dv.
\]
It follows from Theorems \ref{thm2.1}--\ref{thm2.3} that,
\begin{equation}\label{2.21}
\sup_{0\leq t\leq T}|\Sig^{(\nu)}(tN)-\fW_\Gam^{(\nu)}(tN)|=O(N^{\frac \nu 2-\ve})\,\,\,\mbox{a.s.}\,\,\mbox{as}\,\, N\to\infty
\end{equation}
where
\[
|\Sig^{(\nu)}(u)-\fW_\Gam^{(\nu)}(u)|=\max_{i_1,...,i_\nu}|\Sig^{i_1,...,i_\nu}(u)-\fW_\Gam^{i_1,...,i_\nu}(u)|
\]
and we write $\Sig^{(\nu)}(u)=\Sig^{(\nu)}(0,u)$ and $\fW_\Gam^{(\nu)}(u)=\fW_\Gam^{(\nu)}(0,u)$. It is not difficult to see that
(\ref{2.21}) implies (see Section \ref{sec7}) that,
\begin{equation}\label{2.22}
\sup_{0\leq t\leq T}|\Sig^{(\nu)}(t\tau)-\fW_\Gam^{(\nu)}(t\tau)|=O(\tau^{\frac \nu 2-\ve})\,\,\,\mbox{a.s.}\,\,\mbox{as}\,\,\tau\to\infty.
\end{equation}

Next, introduce iterated stochastic integrals defined inductively by
\[
\fW^{(\nu)}(s,t)=\int_s^t\fW^{(\nu-1)}(s,v)\otimes d\cW(v)
\]
where $\fW^{(1)}(s,v)=\cW(s,v)=\cW(v)-\cW(s)$ and we write again $\fW^{(\nu)}(u)=\fW^{(\nu)}(0,u)$.
Coordinate-wise this relation has the form
\[
\fW^{i_1,...,i_\nu}(s,t)=\int_s^t\fW^{i_1,...,i_{\nu-1}}(s,v)d\cW^{i_\nu}(v).
\]
We will show in Section \ref{sec7} that with probability one,
\begin{equation}\label{2.23}
(\tau\log\log\tau)^{-\nu/2}\sup_{0\leq t\leq T}|\fW_\Gam^{(\nu)}(t\tau)-\fW^{(\nu)}(t\tau)|\to 0\,\,\mbox{as}\,\,\tau\to\infty
\end{equation}
which will allow us to deduce a functional law of iterated logarithm for $\Sig^{(\nu)}$.

Namely, write $\cW=\vs^{1/2}\bfW$ where $\bfW=(\bfW_1,...,\bfW_d)$ is the standard $d$-dimensional Brownian motion and $\vs^{1/2}=(\vs_{ij}^{1/2})$
is the square root of the covariance matrix $\vs$. Then
\begin{eqnarray}\label{2.24}
&\fW^{i_1,...,i_\nu}(t)=\int_0^td\cW^{i_\nu}(t_\nu)\int_0^{t_\nu}...\int_0^{t_2}d\cW^{i_1}(t_1)\\
&=\sum_{1\leq j_1,...,j_\nu\leq d}a_{j_1,...,j_\nu}^{i_1,...,i_\nu}\int_0^td\bfW_{j_\nu}(t_\nu)\int_0^{t_\nu}d\bfW_{j_{\nu-1}}(t_{\nu-1})...
\int_0^{t_2}d\bfW_{j_1}(t_1)
\nonumber\end{eqnarray}
where $a_{j_1,...,j_\nu}^{i_1,...,i_\nu}=\vs^{1/2}_{i_\nu j_\nu}\vs^{1/2}_{i_{\nu-1}j_{\nu-1}}\cdots\vs^{1/2}_{i_1j_1}$.

Let $\cH_d$ be the space of all absolute continuous functions $f=(f_1,...,f_d):\,[0,T]\to\bbR^d$ such that $f(0)=0$ and
$\sum_{1\leq i\leq d}\int_0^T(f'_i(t))^2dt<\infty$. Introducing the scalar product for $f,g\in\cH_d$ by
\[
\langle f,g\rangle=\sum_{i=1}^d\int_0^Tf'_i(t)g'_i(t)dt
\]
makes $\cH_d$ a Hilbert space with the norm $|f|_1=\langle f,f\rangle^{1/2}=\| f'\|_{L^2}$. Define also the set $\cC^d$ of all
continuous paths $\gam:\,[0,T]\to\bbR^d$ with the supremum norm and observe that the embedding $\cH_d\to\cC^d$ is compact. Define
\begin{eqnarray*}
&X_\tau^{(\nu)}(t)=(\tau\ln\ln\tau)^{-\nu/2}\fW^{(\nu)}(t\tau),\, t\in[0,T],\,\tau>0\,\,\,\mbox{and}\\
&X^{i_1,...,i_\nu}_\tau(t)=(\tau\ln\ln\tau)^{-\nu/2}\fW^{(i_1,...,i_\nu)}(t\tau).
\end{eqnarray*}
By Proposition 3.1 and Corollary 3.2 from \cite{Bal} we obtain
\begin{proposition}\label{prop2.4} For any $1\leq\nu\leq\nu(M)$ and $1\leq i_1,...,i_\nu\leq d$ the family $\{ X^{i_1,...,i_\nu}_\tau,\,\tau\geq 4\}$
is relatively compact in $\cC^1$ with probability one and its limit set as $\tau\to\infty$ coincides a.s. with the set of all paths $\gam_{i_1,...,i_\nu}:\,[0,T]\to\bbR$
of the form
\begin{equation}\label{2.25}
\gam_{i_1,...,i_\nu}(t)=\sum_{1\leq j_1,...,j_\nu\leq d}a_{j_1,...,j_\nu}^{i_1,...,i_\nu}\int_0^tdf_{j_\nu}(t_\nu)\int_0^{t_\nu}df_{j_{\nu-1}}(t_{\nu-1})...
\int_0^{t_2}df_{j_1}(t_1)
\end{equation}
where $f$ varies in $\{ f\in\cH_d:\,\frac 12|f|^2_1\leq T\}$. Correspondingly, $\{ X_\tau^{(\nu)},\,\tau\geq 4\}$ is relatively compact in $\cC^{d^\nu}$
with probability one
 and its limit set as $\tau\to\infty$ coincides a.s. with the set of all paths $\gam:\,[0,T]\to\bbR^{d^\nu}$ whose $i_1,...,i_\nu$-th coordinate
 is given by (\ref{2.25}).
 \end{proposition}

 Now combining (\ref{2.22}) and (\ref{2.23}), which will be proved in Section \ref{sec7}, with Proposition \ref{prop2.4} we obtain
 \begin{theorem}\label{thm2.5}
 Set
 \begin{eqnarray*}
&\cX_\tau^{(\nu)}(t)=(\tau\ln\ln\tau)^{-\nu/2}\Sig^{(\nu)}(t\tau),\, t\in[0,T],\,\tau>0\,\,\,\mbox{and}\\
&\cX^{i_1,...,i_\nu}_\tau(t)=(\tau\ln\ln\tau)^{-\nu/2}\Sig^{(i_1,...,i_\nu)}(t\tau).
\end{eqnarray*}
Then the assertions concerning the families $\{ X^{i_1,...,i_\nu}_\tau,\,\tau\geq 4\}$ and $\{ X_\tau^{(\nu)},\,\tau\geq 4\}$ from Proposition \ref{prop2.4}
hold true for the families $\{\cX^{i_1,...,i_\nu}_\tau,\,\tau\geq 4\}$ and $\{ \cX_\tau^{(\nu)},\,\tau\geq 4\}$, as well,  for all
tree setups of Theorems \ref{thm2.1}--\ref{thm2.3}.
 \end{theorem}
\begin{remark}\label{rem2.6} Combining results from \cite{LQZ} (see also \cite{BBAK}) with Theorems \ref{thm2.1}--\ref{thm2.3} it is possible to
extend the above law of iterated logarithm to the one valid in the $p$-variation metric.
\end{remark}

\begin{remark}\label{rem2.7}
Theorems \ref{thm2.1}--\ref{thm2.3} can be extended with essentially the same proof to the setup where in place of one discrete
or continuous time process $\xi(k),\, k\in\bbZ$ or $\xi(t),\, t\in\bbR$ we have sequences of $\bbR^d$-valued jointly stationary
processes $\xi^{(i)}(k),\, k\in\bbZ$ or $\xi^{(i)}(t),\, t\in\bbR,\, i=1,2,...$ all satisfying the conditions of the corresponding
Theorems \ref{thm2.1}--\ref{thm2.3} with respect to the same family of $\sig$-algebras $\cF_{mn}$ or $\cF_{st}$. In this setup we
will have iterated sums and integrals of the form
\[
\bbS_N^{(\nu)}(s,t)=N^{-\nu/2}\sum_{[sN]\leq k_1<...<k_\nu<[tN]}\xi^{(1)}(k_1)\otimes\xi^{(2)}(k_2)\otimes\cdots\otimes\xi^{(\nu)}(k_\nu)
\]
and
\[
\bbS_N^{(\nu)}(s,t)=N^{-\nu/2}\int_{sN\leq u_1\leq...\leq u_\nu<tN}\xi^{(1)}(u_1)\otimes\xi^{(2)}(u_2)\otimes\cdots\otimes\xi^{(\nu)}(u_\nu)du_1...du_\nu.
\]
The limiting processes $\bbW_N^{(\nu)}$ should be constructed now not by one Brownian motion $W_N$ having a covariance matrix $\vs$ but by
a sequence of Brownian motions $W_N^{(i)},\, i=1,2,...$ with covariance matrices $\vs^{(i)}$ so that $\bbW_N^{(1)}=W_N^{(1)}$,
\[
\bbW_N^{(2)}(s,t)=\int_s^tW_N^{(1)}(s,v)\otimes dW_N^{(2)}(v)+(t-s)\Gam^{(2)},
\]
where $\Gam^{(i)}=\sum_{l=1}^\infty E(\xi^{(i-1)}(0)\otimes\xi^{(i)}(l))$, and recursively,
\[
\bbW_N^{(\nu)}(s,t)=\int_s^tW_N^{(\nu-1)}(s,v)\otimes dW_N^{(\nu)}(v)+\int_s^t\bbW_N^{(\nu-2)}(s,v)\otimes\Gam^{(\nu)}dv.
\]
Theorem \ref{thm2.5} extends to this setup, as well, with corresponding modifications in the definition of the limiting compact set.
In particular, this will work when the processes $\xi^{(i)}$ are constructed by one appropriate discrete or continuous time
dynamical system $F^n$ or $F^t$ but with different functions $g^{(i)}$, i.e. $\xi^{(i)}(n)=g^{(i)}\circ F^n$ or $\xi^{(i)}(t)
=g^{(i)}\circ F^t$.
\end{remark}

\subsection{An almost sure central limit theorem}\label{subsec2.5}

Employing a slight modification of the method from \cite{LP} we will derive as a corollary of Theorems \ref{thm2.1}--\ref{thm2.3} the following almost sure central limit theorem for iterated sums and integrals
 which was derived earlier for usual sums in \cite{Bro}. Namely we have,
 
\begin{theorem}\label{thm2.8} Let $\bbS_N^{(\nu)}=\bbS_N^{(\nu)}(t,\om)$ and $\bbW_N^{(\nu)}=\bbW_N^{(\nu)}(t,\om)$,
$t\in[0,T],\,\om\in\Om$ be as in one of Theorems \ref{thm2.1}--\ref{thm2.3} above. Assume that $\bbS_N^{(\nu)}(t,\om)$ appearing in Theorems \ref{thm2.2} and \ref{thm2.3} are continuous in $t$ for $P$-almost all $\om$, which holds true when, for instance, the corresponding stationary process $\xi$ is almost surely bounded. Recall, that $\bbS_N^{(\nu)}(t,\om)$ 
from Theorem \ref{thm2.1} is piecewise constant in $t$ with jumps at $t=\frac kN,\, k=0,1,...$ which is also sufficient for 
our arguments. Denote by $\cH_T(\bbR^{d\otimes\nu})$ the set of all paths $\gam(t),\, t\in[0,T],\,\gam(0)=0=(0,...,0)$
on the $\nu$-times tensor product $\bbR^{d\otimes\nu}=\bbR^d\otimes\cdots\otimes\bbR^d$, which is the $\nu d$-dimensional
space, such that $\gam(t)$ may have discontinuities only at the points $t=\frac kN,\, k=0,1,...,N,\, N=1,2,3,...$. Then
for $P$-almost all $\om$ the measures
\[
(\log N)^{-1}\sum_{1\le n\leq N}n^{-1}\del(\bbS_n^{(\nu)}(\cdot,\om))
\]
on $\cH_T(\bbR^{d\otimes\nu})$, where $\del(x)$ is the unit mass at $x\in\cH_T(\bbR^{d\otimes\nu})$, converge weakly 
as $N\to\infty$ to the distribution of $\bbW_1^{(\nu)}$.
\end{theorem}

\section{Auxiliary estimates}\label{sec3}\setcounter{equation}{0}
\subsection{General estimates}\label{susec3.1}

We will need the following estimates which were proved in \cite{Ki23} as Lemma 3.2 and under more restricted assumptions in \cite{FK}, Lemma 3.4.
\begin{lemma}\label{lem3.2}
Let $\eta(k),\,\zeta_k(l),\, k=0,1,...,\, l=0,...,k$ be sequences of random variables on a probability space
$(\Om,\cF,P)$ such that for all $k,l,n\geq 0$ and some $L,M\geq 1$ and $K=\max(2L,4M)$ satisfying (\ref{2.4}),
\begin{eqnarray}\label{3.2}
&\quad\|\eta(k)-E(\eta(k)|\cF_{k-n,k+n})\|_{K},\,\|\zeta_k(l)-E(\zeta_k(l)|\cF_{l-n,l+n})\|_{K}\\
&\leq\be(K,n),\,\,\|\eta(k)\|_{K},\,\|\zeta_k(l)\|_{K}\leq\gam_{K}<\infty\,\,\mbox{and}\,\,E\eta(k)=E\zeta_k(l)=0,\,\forall k,l\nonumber
\end{eqnarray}
where $\be(K,\cdot)$ and $\vp_{L,4M}(\cdot)$ satisfy (\ref{2.4}) and the $\sig$-algebras $\cF_{kl}$ are the same as in
Section \ref{sec2}. Then for any $N\geq 1$,
\begin{equation}\label{3.3}
E\max_{1\leq n\leq N}(\sum_{k=0}^n\eta(k))^{4M}\leq C^{\eta}(M)N^{2M}
\end{equation}
and
\begin{eqnarray}\label{3.4}
&E\max_{m\leq n\leq N}\big(\sum_{k=m}^n\sum_{j=\ell(k)}^{k-1}(\eta(k)\zeta_k(j)-E(\eta(k)\zeta_k(j)))\big)^{2M}\\
&\leq C^{\eta,\zeta}(M)(N-m)^{M}\max_{m\leq k\leq N}(k-\ell(k))^{M}\nonumber
\end{eqnarray}
where $0\leq\ell(k)<k$ is an integer valued function (may be constant) and $C^{\eta,\zeta}(M)>0$ depends only on
$\be,\gam,\vp$ and $M$ but it does not depend on $N,m,\ell$ or on the sequences $\eta$ and $\zeta$ themselves.
\end{lemma}

In the estimates for variational norms we will use the following extended version of the Kolmogorov theorem on the
 H\" older continuity of sample paths which was proved in \cite{Ki23}. For $\nu=2$ this result was proved in Theorem 3.1 of \cite{FH}.

\begin{theorem}\label{hoelder}
For $1\leq\ell<\infty$ let $\bbX^{(\nu)}=\bbX^{(\nu)}(s,t),\,\nu=1,2,...,\ell;\, s\leq t$ be a two parameter
stochastic process which is a multiplicative functional in the sense of \cite{Lyo} and \cite{LQ}, i.e.
$\bbX^{(1)}(s,t)=\bbX^{(1)}(0,t)-\bbX^{(1)}(0,s)$ and for any $\nu=2,...,\ell$ and $0\leq s\leq u\leq t\leq T$
the Chen relations,
\begin{equation}\label{ho1}
\bbX^{(\nu)}(s,t)=\bbX^{(\nu)}(s,u)+\sum_{k=1}^{\nu-1}\bbX^{(k)}(s,u)\otimes\bbX^{(\nu-k)}(u,t)+\bbX^{(\nu)}(u,t)
\end{equation}
hold true,
where $\bbX^{(1)}(s,t)$ is a $d$-dimensional vector (or a vector in a Banach space) and $\otimes$ is a tensor
product in the corresponding space. Let $\bfM\geq\ell$ be an integer, $\be>1/\bfM$ and assume that whenever $0\leq s\leq t
\leq T<\infty$ and $1\leq\nu\leq\ell$,
\begin{equation}\label{ho2}
E\|\bbX^{(\nu)}(s,t)\|^\bfM\leq\bbC_{\ell,\bfM}|t-s|^{\nu\bfM\be}
\end{equation}
for some constants $\bbC_{\ell,\bfM}<\infty$, where $\|\cdot\|$ is the norm in a corresponding space satisfying
$\| a\otimes b\|\leq \| a\|\| b\|$. Then for all $\al\in[0,\be-\frac 1\bfM)$ there exists a modification of
$(\bbX^{(\nu)},\,\nu=1,...,\ell)$, which will have the same notation, and random variables $\bbK_\al^{(\nu)},
\,\nu=1,...,\ell$ such that
\begin{equation}\label{ho3}
E(\bbK_\al^{(\nu)})^{\bfM/\nu}<\infty\quad\mbox{and}\quad\|\bbX^{(\nu)}(s,t)\|\leq\bbK_\al^{(\nu)}|t-s|^{\nu\al}
\quad\mbox{a.s.}
\end{equation}
\end{theorem}

In the discrete time case we will use the following corollary from the above theorem also proved in \cite{Ki23}.
A weaker statement under more restricted conditions was proved in Lemma 4.2 of \cite{FK}. See also Proposition 6.17
in \cite{FZ}.
\begin{proposition}\label{hoelder2}
For each integer $N\geq 1$ let $(X_N,\bbX_N)=(X_N(s,t),\,\bbX_N(s,t))$ be a pair of two parameter processes
in $\bbR^d$ and $\bbR^d\otimes\bbR^d$, respectively, defined for each $s=k/N$ and $t=l/N$ where $0\leq k\leq l
\leq TN$, $T>0$ and $X_N(s,s)=\bbX_N(s,s)=0$. Assume that the Chen relation
\begin{equation}\label{hoe1}
\bbX_N(\frac kN,\frac mN)=\bbX_N(\frac kN,\frac lN)+X_N(\frac kN,\frac lN)\otimes X_N(\frac lN,\frac mN)
+\bbX_N(\frac lN,\frac mN)
\end{equation}
and $X_N(\frac kN,\frac lN)=X_N(0,\frac lN)-X_N(0,\frac kN)$ holds true for any $0\leq k\leq l\leq m\leq TN$.
Suppose that for all $0\leq k\leq l\leq NT$ and some $\bfM\geq 1$, $\ka>\frac 1{\bfM}$, $C(\bfM)>0$ and
$\bbC(\bfM)>0$ (which do not depend on $k,l$ and $N$),
\begin{equation}\label{hoe2}
E|X_N(\frac kN,\frac lN)|^{\bfM}\leq C(\bfM)|\frac {l-k}N|^{\bfM\ka}\,\,\mbox{and}\,\,
E|\bbX_N(\frac kN,\frac lN)|^{\bfM}\leq \bbC(\bfM)|\frac {l-k}N|^{2\bfM\ka}.
\end{equation}
Then for any $\be\in(\frac 1{\bfM},\ka)$ there exist random variables $C_{\ka,\be,N}>0$ and
$\bbC_{\ka,\be,N}>0$ such that
\begin{equation}\label{hoe3}
|X_N(\frac kN,\frac lN)|\leq C_{\ka,\be,N}|\frac {l-k}N|^{\ka-\be}\,\,\mbox{and}\,\,
|\bbX_N(\frac kN,\frac lN)|\leq \bbC_{\ka,\be,N}|\frac {l-k}N|^{2(\ka-\be)}\,\,\mbox{a.s.}
\end{equation}
and
\begin{equation}\label{hoe4}
EC^{\bfM}_{\ka,\be,N}\leq K_{\ka,\be}(\bfM)<\infty\,\,\,\mbox{and}\,\,\, E\bbC^{\bfM/2}_{\ka,\be,N}\leq\bbK_{\ka,\be}(\bfM)<\infty
\end{equation}
where $K_{\ka,\be}(\bfM)>0$ and $\bbK_{\ka,\be}(\bfM)>0$ do not depend on $N$.
\end{proposition}

\subsection{The matrix $\vs$, characteristic functions and partition into blocks}\label{subsec3.2}

In the following lemma, which justifies the definition of the matrix $\vs$ and gives the necessary convergence estimates, was proved in \cite{Ki23}, as well.
Under more restricted conditions this assertion was proved as Lemma 3.4 in \cite{Ki22}.
\begin{lemma}\label{lem3.3}
For each $i,j=1,...,d$ the limit
\begin{equation}\label{3.12}
\vs_{ij}=\lim_{n\to\infty}\frac 1n\sum_{k=0}^n\sum_{m=0}^nE(\xi_i(k)\xi_j(m))
\end{equation}
exists and for any $m,n\geq 0$,
\begin{eqnarray}\label{3.13}
&|n\vs_{ij}-\sum_{k=m}^{m+n}\sum_{l=m}^{m+n}E(\xi_i(k)\xi_j(l))|\\
&\leq 6\sum_{k=0}^n\sum_{l=k+1}^\infty\big((\|\xi_i(0)\|_K+\|\xi_j(0)\|_K)\be(K,l)\nonumber\\
&+(\|\xi_i(0)\|_K\|\xi_j(0)\|_K)\vp_{L,4M}(l)\big).\nonumber
\end{eqnarray}
Similarly, the limit
\[
\Gam^{ij}=\lim_{n\to\infty}E\bbS_N(n)=\lim_{n\to\infty}\frac 1n\sum_{l=1}^n\sum_{m=0}^{l-1}E(\xi_i(m)\xi_j(l))=
\sum_{l=1}^\infty E((\xi_i(0)\xi_j(l))
\]
exists, in particular, $\vs_{ij}=E(\xi_i(0)\xi_j(0))+2\Gam^{ij}$, and
\begin{eqnarray*}
&|n\Gam^{ij}-\sum_{l=1}^{n}\sum_{m=0}^{l-1}E(\xi_i(m)\xi_j(l))|\\
&\leq 3\sum_{k=0}^n\sum_{l=k+1}^\infty\big((\|\xi_i(0)\|_K+\|\xi_j(0)\|_K)\be(K,l)+(\|\xi_i(0)\|_K\|\xi_j(0)\|_K)
\vp_{L,4M}(l)\big).\nonumber
\end{eqnarray*}
\end{lemma}

Next, for each $n\geq 1$ and $w\in\bbR^d$ introduce the characteristic function
 \[
 f_n(w)=E\exp(i\langle w,\, n^{-1/2}\sum_{k=0}^{n-1}\xi(k)\rangle),\, w\in\bbR^d
 \]
 where $\langle\cdot,\cdot\rangle$ denotes the inner product. We will need the following estimate.
 \begin{lemma}\label{lem3.4}
 For any $n\geq 1$,
 \begin{equation}\label{3.15}
 |f_n(w)-\exp(-\frac 12\langle\vs w,\, w\rangle)|\leq C_fn^{-\wp}
 \end{equation}
 for all $w\in\bbR^d$ with $|w|\leq n^{\wp/2}$ where the matrix $\vs$ is given
 by (\ref{2.6}), we can take $\wp\leq\frac 1{20}$ and a constant $C_f>0$ does not depend on $n$.
 \end{lemma}
  \begin{proof} First, observe that (\ref{2.4}) implies that for any $n\geq 1$,
 \begin{equation}\label{3.16}
 \be(K,n)\leq C_1n^{-4}\quad\mbox{and}\quad\vp_{L,4M}(n)\leq C_1n^{-2}
 \end{equation}
 where $C_1>0$ does not depend on $n$. Indeed, $\vp_{L,4M}(n)$ is nonnegative and does not increase
 in $n$. Hence, for some $C>0$ and any $n\geq 1$,
 \begin{eqnarray*}
 &\infty>C\geq\sum_{k=0}^\infty\sum_{l=k+1}^\infty\varpi_{L,4M}(l)\geq\sum_{k=0}^n\sum_{l=k+1}^n\varpi_{L,4M}(l)\\
 &\geq\varpi_{L,4M}(n)\sum_{k=0}^n(n-k)=\frac 12n(n+1)\varpi_{L,4M}(n),
 \end{eqnarray*}
 and so (\ref{2.4}) implies the second estimate in (\ref{3.16}). The same argument applied to $\sup_{l\geq n}\be(K,l)$
 yields the first estimate in (\ref{3.16}). Observe also that for any pair of $L^\infty$ functions $g$ and $h$ and a pair
 of $\sig$-algebras $\cG,\cH\subset\cF$ if $g$ is $\cG$-measurable and $h$ is $\cH$-measurable then
 \[
 |E(gh)-EgEh|=|E((E(g|\cH)-Eg)h)|\leq\| g\|_\infty\| h\|_\infty\vp_{L,4M}(\cH,\cG).
 \]
 Now relying on (\ref{3.16}) and the last observation, which replaces here Lemma 3.1 in \cite{Ki22} used in Lemma 3.10 there,
 the proof of the present lemma proceeds essentially in the same way as the proof of Lemma 3.10 in \cite{Ki22} replacing there
  the $\phi$-mixing coefficient $\phi(n)$ by $\vp_{L,4M}(n)$, the approximation coefficient $\rho(n)$ by $\be(K,n)$ and using
  here Lemmas \ref{lem3.2} and \ref{lem3.3} above in place of Lemmas 3.2 and 3.4 in \cite{Ki22}, respectively. Taking all these
  into account the details of the estimates can be easily reproduced from Lemma 3.10 of \cite{Ki22}.
 \end{proof}

  Next, we introduce blocks of high polynomial power length with gaps between
 them. Set $m_0=0$ and recursively $n_k=m_{k-1}+[k^\rho],\, m_k=n_k+[k^{\rho/4}],\, k=1,2,...$ where $\rho>0$ is
 big and will be chosen later on. Now we define sums
 \begin{eqnarray*}
 &Q_k=\sum_{m_{k-1}\leq j<n_k}E\big(\xi(j)|\cF_{j-\frac 13[(k-1)^{\rho/4}],j+\frac 13[(k-1)^{\rho/4}]}\big)\\
 &\mbox{and}\quad R_k=\sum_{n_k\leq j<m_k}\xi(j),\, k=1,2,...
 \end{eqnarray*}
where the first sums play the role of blocks while the second ones are gaps whose total contribution turns out to
be negligible for our purposes. Set also $\ell_N(t)=\max\{ k:\, m_k\leq Nt\}$ and $\ell_N=\ell_N(T)$. As in \cite{FK},
but under the present more general conditions, we obtain
\begin{lemma}\label{lem3.6} With probability one for all $N\geq 1$,
\begin{equation}\label{3.17}
\sup_{0\leq t\leq T}|S_N(t)-N^{-1/2}\sum_{1\leq k\leq\ell_N(t)}Q_k|=O(N^{-\del})
\end{equation}
provided $\del>0$ is small and $\rho>0$ is large enough.
\end{lemma}
\begin{proof}
Denote the left hand side of (\ref{3.17}) by $I$, then
\[
I\leq\sup_{0\leq t\leq T}I_1(t)+\sup_{0\leq t\leq T}I_2(t)+\sup_{0\leq t\leq T}I_3(t).
\]
Here,
\begin{eqnarray*}
&I_1(t)=N^{-1/2}\big\vert\sum_{1\leq k\leq\ell_N(t)}\big(\sum_{m_{k-1}\leq j<n_k}(\xi(j)\\
&-E(\xi(j)|\cF_{j-\frac 13[(k-1)^{\rho/4}],j+\frac 13[(k-1)^{\rho/4}]})\big)\big\vert,
\end{eqnarray*}
\begin{eqnarray*}
&I_2(t)=N^{-1/2}\big\vert\sum_{1\leq k\leq\ell_N(t)}\sum_{n_{k}\leq j<m_k}\xi(j)\big\vert\\
&\mbox{and}\,\,\, I_3(t)=N^{-1/2}\big\vert\sum_{m_{\ell_N(t)}\leq j<[Nt]}\xi(j)\big\vert .
\end{eqnarray*}
Then
\begin{eqnarray*}
&E\sup_{0\leq t\leq T}I_1^{2M}(t)\leq N^{-M}\ell_N^{2M-1}\sum_{1\leq k\leq\ell_N}k^{(2M-1)\rho}\sum_{m_{k-1}\leq j<n_k}E|\xi(j)\\
&-E(\xi(j)|\cF_{j-\frac 13[(k-1)^{\rho/4}],j+\frac 13[(k-1)^{\rho/4}]})|^{2M}\\
&\leq N^{-M}\ell_N^{2M-1}\sum_{1\leq k\leq\ell_N}k^{(2M-1)\rho}\be^{2M}(K,\frac 13[(k-1)^{\rho/4}])\leq C_2N^{-M(1-\frac 2{\rho+1})}
\end{eqnarray*}
since
\begin{equation}\label{3.18}
(\frac {TN}2)^{\frac 1{\rho+1}}-2\leq\ell_N\leq (TN(\rho+1))^{\frac 1{\rho+1}}
\end{equation}
and by (\ref{3.16}),
\[
C_2=(T(\rho+1))^{\frac 1{\rho+1}}\max_{k\geq 1}(k^\rho\be(K,\frac 13[(k-1)^{\rho/4}]))<\infty.
\]
By the Chebyshev inequality
\begin{equation*}
P\{\sup_{0\leq t\leq T}I_1(t)>N^{-\del}\}\leq N^{2M\del}E\sup_{0\leq t\leq T}I_1^{2M}(t).
\end{equation*}
Choosing $\rho>0$ big and $\del>0$ small enough so that $1-2\del-\frac 2{\rho+1}>1/2$ and taking $M\geq 4$ we obtain that
the right hand side of this inequality is bounded by $N^{-2}$, and so by the Borel-Cantelli lemma
\begin{equation}\label{3.19}
\sup_{0\leq t\leq T}I_1(t)=O(N^{-\del})\quad\mbox{a.s.}
\end{equation}

Next, by Lemma \ref{lem3.2},
\begin{eqnarray*}
&E\sup_{0\leq t\leq T}I_2^{2M}(t)\leq N^{-M}\sum_{0\leq l\leq\ell_N}E\big\vert\sum_{1\leq k\leq l}\sum_{n_{k}\leq j<m_k}\xi(j)\big\vert^{2M}\\
&\leq C_3(M)N^{-M}\sum_{1\leq l\leq\ell_N}(\sum_{1\leq k\leq l}k^{\rho/4})^M\leq C_3(M)N^{-\frac 14(3-\frac 7{\rho+1})M}
\end{eqnarray*}
where $C_3(M)>0$ does not depend on $N$. By the Chebyshev inequality
\begin{equation*}
P\{\sup_{0\leq t\leq T}I_2(t)>N^{-\del}\}\leq N^{2M\del}E\sup_{0\leq t\leq T}I_2^{2M}(t).
\end{equation*}
Choosing $\rho>0$ big and $\del>0$ small enough so that $\frac 7{\rho+1}+8\del\leq 1$ and taking $M\geq 4$ we obtain that
the right hand side of this inequality is bounded by $N^{-2}$, and so by the Borel-Cantelli lemma
\begin{equation}\label{3.20}
\sup_{0\leq t\leq T}I_2(t)=O(N^{-\del})\quad\mbox{a.s.}
\end{equation}

Next,
\begin{eqnarray*}
&\sup_{0\leq t\leq T}I_3^{2M}(t)=N^{-M}\max_{1\leq k\leq\ell_N}\max_{m_k\leq j<m_{k+1}\wedge Nt}|\sum_{m_k\leq i
\leq j}\xi(i)|^{2M}\\
&\leq N^{-M}\sum_{1\leq k\leq\ell_N}\sum_{m_k\leq j<m_{k+1}}|\sum_{m_k\leq i\leq j}\xi(i)|^{2M}.\\
\end{eqnarray*}
By Lemma \ref{lem3.2},
\[
E|\sum_{m_k\leq i< j}\xi(i)|^{2M}\leq C_4(M)(j-m_k)^M,
\]
and so by (\ref{3.18}),
\[
E\sup_{0\leq t\leq T}I_3^{2M}(t)\leq C_4(M)N^{-M}\sum_{k=1}^{\ell_N}(m_{k+1}-m_k)^{M+1}\leq C_5(M)N^{-\frac M{\rho+1}+1}
\]
where $C_4(M)>0$ and $C_5(M)>0$ do not depend on $N$. By the Chebyshev inequality
\begin{equation*}
P\{\sup_{0\leq t\leq T}I_3(t)>N^{-\del}\}\leq N^{2M\del}E\sup_{0\leq t\leq T}I_3^{2M}(t).
\end{equation*}
Choosing $\del\leq\frac 1{4(\rho+1)}$ and $M\geq 12(\rho+1)$ we bound the right hand side of this inequality by $N^{-2}$ which
together with the Borel-Cantelli lemma yields that
\begin{equation}\label{3.21}
\sup_{0\leq t\leq T}I_3(t)=O(N^{-\del})\quad\mbox{a.s.}
\end{equation}
Finally, (\ref{3.17}) follows from (\ref{3.19})--(\ref{3.21}).
\end{proof}

Next, set
\begin{equation*}
\cG_k=\cF_{-\infty,n_k+\frac 13[k^{\rho/4}]},
\end{equation*}
and so $Q_k$ is $\cG_k$-measurable.
As a corollary we obtain
\begin{lemma}\label{lem3.7} For any $k\geq 1$,
\begin{eqnarray}\label{3.22}
&E|E(\exp(i\langle w,\,(n_k-m_{k-1})^{-1/2}Q_k\rangle |\cG_{k-1})-\exp(-\frac 12\langle\vs w,w\rangle)|\\
&\leq \vp_{L,4M}(\frac 13[(k-1)^{\rho/4}])+k^{\rho/2(1+\wp/2)}\be(K,\frac 13[(k-1)^{\rho/4}])+C_f[k^\rho]^{-\wp}\nonumber
 \end{eqnarray}
 for all $w\in\bbR^d$ with $|w|\leq(n_k-m_{k-1})^{\wp/2}$ where $C_f>0$ is the same as in (\ref{3.15}).
 \end{lemma}
 \begin{proof} Set
 \[
 F_k=\exp(i\langle w,\,(n_k-m_{k-1})^{-1/2}Q_k\rangle).
 \]
 Then $F_k$ is $\cF_{m_{k-1}-\frac 13[(k-1)^{\rho/4}],\infty}$-measurable and since $|F_k|=1$ we obtain by (\ref{2.1}) and (\ref{2.2}) that
 \[
 E|E(F_k|\cG_{k-1})-EF_k|\leq\| E(F_k|\cG_{k-1})-EF_k\|_{4M}\leq \vp_{L,4M}(\frac 13[(k-1)^{\rho/4}]).
 \]
 Since $|e^{ia}-e^{ib}|\leq |a-b|$, we obtain from (\ref{2.3})
 taking into account the stationarity of $\xi(k)$'s that,
 \[
 |EF_k-f_{n_k-m_{k-1}}(w)|\leq |w|k^{\rho/2}\be(K,\frac 13[(k-1)^{\rho/4}]),
 \]
 and (\ref{3.22}) follows from Lemma \ref{lem3.4}.
 \end{proof}

 \subsection{Strong approximations}\label{subsec3.3}

We will rely on the following result which is a version of Theorem 1 in \cite{BP} with some features
taken from Theorem 4.6 in \cite{DP} (see also Theorem 3 in \cite{MP}).
\begin{theorem}\label{thm3.8} Let $\{ V_k,\, k\geq 1\}$ be a sequence of random vectors with values in $\bbR^d$
defined on some probability space $(\Om,\cF,P)$ and such that $V_k$ is measurable with respect to $\cG_k,\, k=1,2,...$
where $\cG_k,\, k\geq 1$ is a filtration of countably generated sub-$\sig$-algebras of $\cF$. Assume that the probability
space is rich enough so that there exists on it a sequence of uniformly distributed on $[0,1]$ independent random variables
$U_k,\, k\geq 1$ independent of $\vee_{k\geq 1}\cG_k$. Let $G$ be a probability distribution on $\bbR^d$ with the characteristic
function $g$. Suppose that for some nonnegative numbers $\nu_m,\del_m$ and $K_m\geq 10^8d$,
 \begin{equation}\label{3.15+}
 E\big\vert E(\exp(i\langle w,V_k\rangle)|\cG_{k-1})-g(w)\big\vert \leq\nu_k
 \end{equation}
 for all $w\in\bbR^d$ with $|w|\leq K_k$ and
 \begin{equation}\label{3.16+}
 G\{ x:\, |x|\geq\frac 14K_k\}\leq\del_k.
 \end{equation}
 Then there exists a sequence $\{ W_k,\, k\geq 1\}$ of $\bbR^d$-valued random vectors defined on
 $(\Om,\cF,P)$ with the properties

 (i) $W_k$ is $\cG_k\vee\sig\{U_k\}$-measurable for each $k\geq 1$;

 (ii) each $W_k,\, k\geq 1$ has the distribution $G$ and $W_k$ is independent of $\sig\{ U_1,...,U_{k-1}\}\vee
 \cG_{k-1}$, and so also of $W_1,...,W_{k-1}$;

 (iii) Let $\vr_k=16K^{-1}_k\log K_k+2\nu_k^{1/2}K_k^d+2\del_k^{1/2}$. Then
 \begin{equation}\label{3.17+}
 P\{ |V_k-W_k|\geq\vr_k\}\leq\vr_k.
 \end{equation}
 \end{theorem}

In order to apply Theorem \ref{thm3.8} we take $V_k=(n_k-m_{k-1})^{-1/2}Q_k,\, \cG_k$ as in Lemma \ref{lem3.7} and
\[
g(w)=\exp(-\frac 12\langle\vs w,w\rangle)
\]
so that $G$ is the mean zero $d$-dimensional Gaussian distribution with the covariance matrix $\vs$. Relying on Lemma \ref{lem3.7} we take $\wp=\frac 1{20}$,
\begin{equation*}
K_k=(n_k-m_{k-1})^{\wp/4d}\leq (n_k-m_{k-1})^{\wp/2}\,\,\mbox{and}\,\,
\nu_k=C_6k^{-\rho\wp}
\end{equation*}
for some $C_6>0$ independent of $k\geq 1$. By the Chebyshev inequality we have also
\begin{eqnarray*}
&G\{ x:\, |x|\geq \frac {K_k}4\}=P\{|\Psi|\geq\frac 14(n_k-m_{k-1})^{\wp/4d}\}\\
&\leq 4d\|\vs\|(n_k-m_{k-1})^{-\wp/2d}\leq C_7k^{-\wp\rho/2d}
\end{eqnarray*}
for some $C_7>0$ which does not depend on $k$.

Now Theorem \ref{thm3.8} provides us with random vectors $W_k,\, k\geq 1$ satisfying the properties (i)--(iii), in
particular, the random vector $W_k$ has the mean zero Gaussian distribution with the covariance matrix $\vs$, it is
independent of $W_1,...,W_{k-1}$ and the property (iii) holds true with
\[
\vr_k=4\frac \wp d(n_k-m_{k-1})^{-\wp/4d}\log(n_k-m_{k-1})+2C_6^{1/2}k^{- \rho\wp/4}+2C_7^{1/2}k^{-\rho\wp/4d}.
\]
Next, we choose $\rho>160d$ which gives
\begin{equation}\label{3.18+}
\vr_k\leq C_8k^{-2}
\end{equation}
for all $k\geq 1$ where $C_8>0$ does not depend on $k$.

Next, let $W(t),\, t\geq 0$ be a $d$-dimensional Brownian motion with the covariance matrix $\vs$ at the time 1. Then the sequence
 of random vectors $\tilde W_k=(n_k-m_{k-1})^{-1/2}(W(n_k)-W(m_{k-1})),\, k=1,2,...$ and $W_k,\, k\geq 1$ have the same
 distributions. Denote by $R$ the joint distribution of the process $\xi(n),\,-\infty<n<\infty$ and of the sequence
   $W_k,\, k\geq 1$ and let $\tilde R$ be the joint distribution of the sequence
   $\tilde W_k,\, k\geq 1$ and a $d$-dimensional Brownian motion $W(t),\, t\geq 0$
   with the covariance matrix $\vs$ at the time 1. Since the second marginal of $R$ coincides with the
   first marginal of $\tilde R$, it follows by Lemma A1 of \cite{BP} that the process $\xi$
   and the sequence $W_k,\, k\geq 1$ can be redefined on a richer probability space where there exists a Brownian motion $W(t),\, t\geq 0$ with the covariance matrix $\vs$ (at the time 1) such
   that $W_k=(n_k-m_{k-1})^{-1/2}(W(n_k)-W(m_{k-1}))$, and so from now on we will rely on
   this equality and assume that these $W_k$'s satisfy (\ref{3.17+}) with $\vr_k$ satisfying
   (\ref{3.18}). It follows by the Borel-Cantelli lemma that there exists a random variable $D=D(\om)<\infty$ a.s. such that
   \begin{equation}\label{3.19+}
   |V_k-W_k|\leq Dk^{-2}\quad\mbox{a.s.}
   \end{equation}

   Now we can obtain the following result.
   \begin{lemma}\label{lem3.9} With probability one,
  \begin{equation}\label{3.20+}
  \sup_{0\leq t\leq T}|\sum_{1\leq k\leq\ell_N(t)}Q_k-W(tN)|=O(N^{\frac 12-\del})
  \end{equation}
  for some $\del>0$ which does not depend on $N$.
  \end{lemma}
  \begin{proof} We have
  \[
  J_N(t)=|\sum_{1\leq k\leq \ell_N(t)}Q_k-W(tN)|\leq  J_N^{(1)}(t)+| J_N^{(2)}(t)|
  \]
  where by (\ref{3.18}) and (\ref{3.19+}),
   \begin{equation}\label{3.21+}
  J_N^{(1)}(t)=\sum_{1\leq k\leq \ell_N(t)}(n_k-m_{k-1})^{1/2}|V_k-W_k|\leq D\sum_{1\leq k\leq \ell_N}
  [k^\rho]^{1/2}k^{-2}\leq D2^{\rho/2}N^{\frac 12(1-\frac 3{\rho+1})}
  \end{equation}
 and
 \[
  J_N^{(2)}(t)=W(tN)-W(m_{\ell_N(t)})+\sum_{1\leq k\leq\ell_N(t)}(W(m_k)-W(n_k)).
  \]
 Observe that  $ J_N^{(2)}(t),\, t\geq 0$ is a sum of independent random vectors and we could employ here the Burkholder--Davis--Gundy
 inequality, but for our purposes the following simpler estimate will also suffice,
 \begin{eqnarray*}
  &E\sup_{0\leq t\leq T}|J^{(2)}_N(t)|^{2M}\leq (\ell_N+2)^{2M}\sum_{1\leq k\leq\ell_N+1}E|W(m_k)-W(n_k)|^{2M}\\
  &\leq C_J(M)(\ell_N+2)^{2M}\sum_{1\leq k\leq\ell_N+1}k^{M\rho/4}\\
  &\leq\tilde  C_J(M)(\ell_N+2)^{\frac M4(\rho+12)}\leq\tilde  C_J(M)(T+2)^{\frac M4(\rho+12)}N^{\frac M4(\frac {\rho+12}{\rho+1})}
  \end{eqnarray*}
  where $C_J(M),\tilde C_J(M)>0$ do not depend on $N$ and we use (\ref{3.18}) on the last step. By the Chebyshev inequality
  \[
  P\{\sup_{0\leq t\leq T}|J^{(2)}_N(t)|\geq N^{\frac 12-\del}\}\leq\hat C_J(M)N^{-M(\frac 23-2\del)}
  \]
  provided $\rho>36$, and so if $M>6$ and $\del<\frac 16$ then the right hand side of this inequality is
  a term of a convergent series in $N$ which together with the Borel--Cantelli lemma and (\ref{3.21+}) yields (\ref{3.20+}).
  \end{proof}

Now combining Lemmas \ref{lem3.6} and \ref{lem3.9} we obtain that for some $\del>0$ and all $N\geq 1$,
\begin{equation}\label{3.22+}
\sup_{0\leq t\leq T}|S_N(t)-W_N(t)|=O(N^{-\del})\quad\mbox{a.s.},
\end{equation}
where $W_N(t)=N^{-1/2}W(tN)$ is another Brownian motion.

\section{Discrete time case estimates}\label{sec4}\setcounter{equation}{0}
\subsection{Variational norm estimates for sums}\label{subsec4.1}

First, we write,
\begin{eqnarray}\label{4.1}
&\| S_N-W_N\|_{p,[0,T]}=\sup_{0=t_0<t_1<...<t_m=T}\big(\sum_{0\leq i<m}|S_N(t_{i+1})\\
&-W_N(t_{i+1})-S_N(t_{i})+W_N(t_i)|^p\big)^{1/p}\leq J^{(1)}_N+J_N^{(2)}+J_N^{(3)}\nonumber
\end{eqnarray}
where $p\in(2,3)$,
\begin{eqnarray*}
&J_N^{(1)}=\sup_{0=t_0<t_1<...<t_m= T}\big(\sum_{i:\, t_{i+1}-t_i>N^{-(1-\al)}}|S_N(t_{i+1})-W_N(t_{i+1})\\
&-S_N(t_i)+W_N(t_i)|^p\big)^{1/p},
\end{eqnarray*}
\[
J_N^{(2)}=\sup_{0=t_0<t_1<...<t_m= T}\big(\sum_{i:\, t_{i+1}-t_i\leq N^{-(1-\al)}}|S_N(t_{i+1})-S_N
(t_i)|^p\big)^{1/p},
\]
\[
J_N^{(3)}=\sup_{0=t_0<t_1<...<t_m= T}\big(\sum_{i:\, t_{i+1}-t_i\leq N^{-(1-\al)}}|W_N(t_{i+1})-W_N
(t_i)|^p\big)^{1/p},
\]
with $\al\in (0,1)$ which will be specified later on. Since there exist no more than $[TN^{1-\al}]$ intervals
$[t_{i},t_{i+1}]$ such that $t_{i+1}-t_i>N^{-(1-\al)}$  and we obtain by (\ref{3.22+}) that
\begin{equation}\label{4.2}
J_N^{(1)}= O(N^{-\del+p^{-1}(1-\al)})\quad\mbox{a.s.}
\end{equation}
where $\al$ will be chosen so that $\al>1-\del$.

In order to estimate $J_N^{(2)}$ and $J_N^{(3)}$ observe that $\sum_{0\leq i<m}(t_{i+1}-t_i)=T$ and
relying on Proposition \ref{hoelder2} for $J_N^{(2)}$ and on Theorem \ref{hoelder} for $J_N^{(3)}$ together
with Lemma \ref{lem3.2} we estimate easily (see Section 4.1 in \cite{Ki23} or Section 3.4 in \cite{FK}),
\[
E(J_N^{(2)})^{4M}\leq K_{S,\be}(M)T^{4M/p}N^{-4M(1-\al)(\frac 12-\frac 1p-\be)}
\]
and
\[
E(J^{(3)}_N)^{4M}\leq K_{W,\be}(M)T^{4M/p}N^{-4M(1-\al)(\frac 12-\frac 1p-\be)}
\]
where $K_{S,\be}(M)>0$ and $K_{W,\be}(M)>0$ do not depend on $N$ and $\be$ is chosen to satisfy $\frac 1{4M}<\be<\frac 12-\frac 1p$.
By the Chebyshev inequality for any $\gam,N>0$,
\[
P\{ J^{(2)}_N\geq N^{-(1-\al)\gam}\}\leq K_{S,\be}(M)T^{4M/p}N^{-4M(1-\al)(\frac 12-\frac 1p-\be-\gam)}
\]
and
\[
P\{ J^{(3)}_N\geq N^{-(1-\al)\gam}\}\leq K_{W,\be}(M)T^{4M/p}N^{-4M(1-\al)(\frac 12-\frac 1p-\be-\gam)}.
\]
Choosing $\be+\gam<\frac 12-\frac 1p$ and $M\geq\frac 12(1-\al)^{-1}(\frac 12-\frac 1p-\be-\gam)^{-1}$ we obtain that the right
hand side of the inequalities above form terms of convergent series, and so by the Borel--Cantelli lemma,
\[
J^{(2)}_N=O(N^{-(1-\al)\gam})\quad\mbox{and}\quad J^{(3)}_N=O(N^{-(1-\al)\gam})\quad\mbox{a.s.}
\]
These together with (\ref{4.2}) proves Theorem \ref{thm2.1} for $\nu=1$.

\subsection{Supremum norm estimates for iterated sums}\label{subsec4.2}

We start with the supremum norm estimates for iterated sums.
Set $m_N=[N^{1-\ka}]$ with a small $\ka>0$ which will be chosen later
on, $\nu_N(l)=\max\{ jm_N:\, jm_N<l\}$ if $l>m_N$ and
\[
R_i(k)=R_i(k,N)=\sum_{l=(k-1)m_N}^{km_N-1}\xi_i(l)\,\,\mbox{for}\,\, k=1,2,...,\iota_N(T)
\]
 where $\iota_N(t)=[[Nt]m_N^{-1}]$.
For $1\leq i,j\leq d$ define
\begin{equation*}
\bbU_N^{ij}(t)=N^{-1}\sum_{l=m_N}^{\iota_N(t)m_N-1}\xi_j(l)\sum_{k=0}^{\nu_N(l)}\xi_i(k)=N^{-1}\sum_{1<l\leq\iota_N(t)}
R_j(l)\sum_{k=0}^{(l-1)m_N-1}\xi_i(k).
\end{equation*}
Set also
\[
\bar\bbS_N^{ij}(t)=\bbS_N^{ij}(t)-t\sum_{l=1}^\infty E(\xi_i(0)\xi_j(l))
\]
where the series converges absolutely in view of Lemma \ref{lem3.3}.
Relying on Lemma \ref{lem3.2} it is not difficult to show (cf. Lemma 4.1 in \cite{FK}) that for all $i,j=1,...,d$ and $N\geq 1$,
\begin{equation}\label{4.3}
E\sup_{0\leq t\leq T}|\bar\bbS_N^{ij}(t)-\bbU_N^{ij}(t)|^{2M}\leq C_{SU}(M)N^{-M\min(\ka,1-\ka)}
\end{equation}
where $C_{SU}(M)>0$ does not depend on $N$. Indeed,
\begin{equation}\label{4.4}
|\bar\bbS_N^{ij}(t)-\bbU_N^{ij}(t)|\leq |\cI_N^{(1)}(t)|+|\cI_N^{(2)}(t)|+|\cI_N^{(3)}(t)|
+|\cI_N^{(4)}(t)|
\end{equation}
where
\[
\cI_N^{(1)}(t)=\cI_N^{1;ij}(t)=N^{-1}\sum_{l=m_N}^{\iota_N(t)m_N-1}\sum_{k=\nu_N(l)+1}^{l-1}\big(\xi_j(l)\xi_i(k)-
E(\xi_j(l)\xi_i(k))\big),
\]
\[
\cI_N^{(2)}(t)=\cI_N^{2;ij}(t)=N^{-1}\sum_{l=\iota_N(t)m_N}^{[Nt]-1}\sum_{k=0}^{l-1}\big(\xi_j(l)\xi_i(k)-
E(\xi_j(l)\xi_i(k))\big),
\]
\[
\cI_N^{(3)}(t)=\cI_N^{3;ij}(t)=N^{-1}\sum_{l=1}^{m_N-1}\sum_{k=0}^{l-1}\big(\xi_j(l)\xi_i(k)-
E(\xi_j(l)\xi_i(k))\big)
\]
and
\[
\cI_N^{(4)}(t)=\cI_N^{4;ij}(t)=\cI_N^{4,1;ij}(t)-\cI_N^{4,2;ij}(t)
\]
with
\[
\cI_N^{4,1;ij}(t)=N^{-1}\sum_{l=1}^{[Nt]-1}\sum_{k=0}^{l-1}E(\xi_j(l)\xi_i(k))-t\sum_{l=1}^\infty E(\xi_i(0)\xi_j(l))
\]
and
\[
\cI_N^{4,2;ij}(t)=N^{-1}\sum_{l=m_N}^{\iota_N(t)m_N-1}\sum_{k=0}^{\nu_N(l)}E(\xi_j(l)\xi_i(k)).
\]
By Lemma \ref{lem3.2},
\begin{eqnarray}\label{4.5}
&E\sup_{0\leq t\leq T}|\cI_N^{(1;i,j)}(t)|^{2M}\leq C_{\cI}(M)T^MN^{-M\ka},\\
& E\sup_{0\leq t\leq T}|\cI_N^{(2;i,j)}(t)|^{2M}\leq C_{\cI}(M)T^MN^{-M\ka}\nonumber\\
&\mbox{and}\,\,\, E\sup_{0\leq t\leq T}|\cI_N^{(3;i,j)}(t)|^{2M}\leq C_{\cI}(M)N^{-2M\ka}\nonumber
\end{eqnarray}
where $C_{\cI}(M)>0$ does not depend on $N$. By the same estimates as in Lemma \ref{lem3.3} we obtain also that
\[
\sup_{0\leq t\leq T}|\cI_N^{(4,1;i,j)}(t)|\leq C_{\cI,4}N^{-1}\,\,\mbox{and}\,\,
\sup_{0\leq t\leq T}|\cI_N^{(4,2;i,j)}(t)|\leq C_{\cI,4}N^{-(1-\ka)}
\]
for some $C_{\cI,4}>0$ which does not depend on $N$. This together with (\ref{4.4}) and (\ref{4.5}) yields (\ref{4.3}).
Choosing $M$ in (\ref{4.3}) large enough and using first the Chebyshev inequality and then the Borel--Cantelli lemma we derive as before
that
\begin{equation}\label{4.6}
\sup_{0\leq t\leq T}|\bar\bbS_N^{ij}(t)-\bbU_N^{ij}(t)|=O(N^{-\del_1})\,\,\mbox{a.s.}
\end{equation}
for some $\del_1>0$ which does not depend on $N$.

Now set
\[
S^i_N(t)=N^{-1/2}\sum_{0\leq k<[Nt]}\xi_i(k),\,\, i=1,...,d
\]
and observe that
\[
\bbU^{ij}_N(t)=\sum_{2\leq l\leq\iota_N(t)}\big(S^j_N\big(\frac {lm_N}N\big)-S^j_N\big(\frac {(l-1)m_N}N\big)\big)S^i_N
\big(\frac {(l-1)m_N}N\big).
\]
Define
\[
\bbV^{ij}_N(t)=\sum_{2\leq l\leq\iota_N(t)}\big(W^j_N\big(\frac {lm_N}N\big)-W^j_N\big(\frac {(l-1)m_N}N\big)\big)W^i_N
\big(\frac {(l-1)m_N}N\big)
\]
where $W_N=(W^1_N,...,W^d_N)$ is the $d$-dimensional Brownian motion with the covariance matrix $\vs$ (at the time 1)
appearing in (\ref{2.6}). Then
\begin{eqnarray}\label{4.7}
&\sup_{0\leq t\leq T}|\bbU^{ij}_N(t)-\bbV^{ij}_N(t)|\leq\sum_{2\leq l\leq\iota_N(T)}\big(\big(\big\vert S^j_N\big(\frac {lm_N}N\big)-W^j_N\big(\frac {lm_N}N\big)\big\vert\\
&+\big\vert S^j_N\big(\frac {(l-1)m_N}N\big)-W^j_N\big(\frac {(l-1)m_N}N\big)\big\vert\big)\big\vert
S^i_N\big(\frac {(l-1)m_N}N\big)\big\vert\nonumber\\
&+\big\vert W^j_N\big(\frac {lm_N}N\big)-W^j_N\big(\frac {(l-1)m_N}N\big)\big\vert\big\vert  S^i_N\big(\frac {(l-1)m_N}N\big)-W^i_N\big(\frac {(l-1)m_N}N\big)\big\vert\big).\nonumber
\end{eqnarray}

By Lemma \ref{lem3.2},
\begin{equation*}
E\max_{2\leq l\leq\iota_N(T)}\big\vert S^i_N\big(\frac {(l-1)m_N}N\big)\big\vert^{2M}
\leq C_{S}(M)T^{M}
\end{equation*}
for some $C_{S}(M)>0$ which does not depend on $N$.
Taking $M$ large enough we obtain by the Chebyshev inequality together with the Borel--Cantelli lemma that for any $\gam>0$,
\begin{equation}\label{4.8}
\max_{2\leq l\leq\iota_N(T)}\big\vert S^i_N\big(\frac {(l-1)m_N}N\big)\big\vert=O(N^\gam)\quad\mbox{a.s.}
\end{equation}
Next, write
\begin{eqnarray*}
&E\max_{2\leq l\leq\iota_N(T)}\big\vert W^j_N\big(\frac {lm_N}N\big)-W^j_N\big(\frac {(l-1)m_N}N\big)\big\vert^{2M}\\
&\leq\sum_{2\leq l\leq\iota_N(T)}E\big\vert W^j_N\big(\frac {lm_N}N\big)-W^j_N\big(\frac {(l-1)m_N}N\big)\big\vert^{2M}.
\end{eqnarray*}
Using the standard moment estimates for the Brownian motion and relying on the Chebyshev inequality and the Borel--Cantelli
lemma we obtain similarly to the above that for $\gam<\ka/2$,
\begin{equation}\label{4.9}
\max_{2\leq l\leq\iota_N(T)}\big\vert W^j_N\big(\frac {lm_N}N\big)-W^j_N\big(\frac {(l-1)m_N}N\big)\big\vert=O(N^{-\gam})\quad\mbox{a.s}.
\end{equation}
Now, combining (\ref{4.7})--(\ref{4.9}) with already proved Theorem \ref{thm2.1} for $\nu=1$ we obtain that
\begin{equation}\label{4.10}
\sup_{0\leq t\leq T}|\bbU^{ij}_N(t)-\bbV^{ij}_N(t)|=O(N^{-\del_2})\quad\mbox{a.s.}
\end{equation}
where $\del_2=\del-\ka-\gam$ and we choose $\ka$ and $\gam$ so small that $\del_2>0$.

Next, observe that
\begin{eqnarray}\label{4.11}
&\sup_{0\leq t\leq T}|\int_0^tW_N^i(s)dW^{j}_N(s)-\bbV^{ij}_N(t)|\\
&\leq\sup_{0\leq t\leq T}|J_N^{(1;i,j)}(t)|+\sup_{0\leq t\leq m_NN^{-1}}|J_N^{(2;i,j)}(t)|+\sup_{0\leq t\leq T}|J_N^{(3;i,j)}(t)|
\nonumber\end{eqnarray}
where
\[
J_N^{(1;i,j)}(t)=\sum_{2\leq l\leq\iota_N(t)}\int_{(l-1)m_NN^{-1}}^{lm_NN^{-1}}\big(W^i_N(s)-
W^i_N\big(\frac {(l-1)m_N}N\big)\big)dW_N^j(s),
\]
\[
J_N^{(2;i,j)}(t)=\int_0^{t}W_N^i(s)dW^j_N(s)\,\,\mbox{and}\,\,
J_N^{(3;i,j)}(t)=\int_{(\iota_N(t)\vee 1)m_NN^{-1}}^{t}W_N^i(s)dW^j_N(s).
\]
By the standard (martingale) moment estimates for stochastic integrals (see, for instance, \cite{Mao}, Sections 1.3 and 1.7),
\begin{eqnarray*}
&E\sup_{0\leq t\leq T}|J_N^{(1;i,j)}(t)|^{2M}\\
&\leq (\frac {2M}{2M-1})^{2M}E|\int_{m_NN^{-1}}^T(W_N^i(s)-W^i_N([\frac {sN}{m_N}]\frac {m_N}N))dW^j_N(s)|^{2M}\\
&\leq(\frac {2M}{2M-1})^{3M}T^{M-1}\int_0^TE|W_N^i(s)-W^i_N([\frac {sN}{m_N}]\frac {m_N}N)|^{2M}ds\leq C^{(1)}_{J}(M,T)N^{-M\ka}\\
\end{eqnarray*}
where $C^{(1)}_{J}(M,T)>0$ does not depend on $N$. Similarly,
\begin{eqnarray*}
&E\sup_{0\leq t\leq T}|J_N^{(2;i,j)}(t)|^{2M}\leq C^{(2)}_{J}(M,T)N^{-M\ka}\,\,\mbox{and}\,\, E\sup_{0\leq t\leq T}|J_N^{(3;i,j)}(t)|^{2M}\\
&\leq\sum_{2\leq l\leq\iota_N(T)}E\sup_{0\leq u\leq m_N}|\int_{(l-1)m_NN^{-1}}^{((l-1)m_N+u)N^{-1}}W_N^i(s)dW_N^j(s)|^{2M}\\
&\leq C^{(3)}_{j}(M,T)N^{-(M-1)\ka}
\end{eqnarray*}
where $C^{(2)}_{J}(M,T),\, C^{(3)}_{J}(M,T) >0$ do not depend on $N$. Taking $\del_3<\frac 12\ka$ and $M\geq(2+\ka)(\ka-2\del_3)^{-1}$ and
employing the Chebyshev inequality together with the Borel--Cantelli lemma in the same way as above, we conclude that
\begin{equation*}
\sup_{0\leq t\leq T}|J_N^{(1;i,j)}(t)|+\sup_{0\leq t\leq T}|J_N^{(2;i,j)}(t)|+\sup_{0\leq t\leq T}|J_N^{(3;i,j)}(t)|=
O(N^{-\del_3})\,\,\,\mbox{a.s.}
\end{equation*}
This together with (\ref{4.6}), (\ref{4.10}) and (\ref{4.11}) completes the proof of Theorem 2.1 for $\nu=2$ in the supremum norm.

\subsection{Variational norm estimates for iterated sums}\label{subsec4.3}

For $0\leq s<t\leq T$ and $i,j=1,...,d$ set
\[
\bar\bbS_N^{ij}(s,t)=N^{-1}\sum_{[sN]\leq k<l<[Nt]}\xi_i(k)\xi_j(l)-(t-s)\sum_{l=1}^\infty E(\xi_i(0)\xi_j(l))
\]
and
\[
\bar\bbW_N^{ij}(s,t)=\int_s^t(W^i_N(u)-W^i_N(s))dW_N^j(u).
\]
Hence,
\[
\bbS_N^{ij}(s,t)-\bbW_N^{ij}(s,t)=\bar\bbS_N^{ij}(s,t)-\bar\bbW_N^{ij}(s,t).
\]
Now for $0\leq T$,
\begin{eqnarray}\label{4.12}
&\|\bbS_N^{ij}-\bbW_N^{ij}\|_{p/2,[0,T]}=\sup_{0=t_0<t_1<...<t_m=T}(\sum_{0\leq q<m} |\bbS_N^{ij}(t_q,t_{q+1})\\
&- \bbW_N^{ij}(t_q,t_{q+1})|^{p/2})^{1/p}\leq J_N^{(1)}+2^{p/2}(J_N^{(2)}+J_N^{(3)})\nonumber
\end{eqnarray}
where $p\in(2,3)$,
\begin{equation*}
J_N^{(1)}=\sup_{0=t_0<t_1<...<t_m= T}\big(\sum_{l:\, t_{l+1}-t_l>N^{-(1-\al)}}|\bbS_N^{ij}(t_l,t_{l+1})-
\bbW_N^{ij}(t_l,t_{l+1})|^{p/2}\big)^{1/p},
\end{equation*}
\[
J_N^{(2)}=\sup_{0=t_0<t_1<...<t_m= T}\big(\sum_{l:\, t_{l+1}-t_l\leq N^{-(1-\al)}}|\bar\bbS_N^{ij}(t_l,t_{l+1})|
^{p/2}\big)^{1/p}
\]
and
\[
J_N^{(3)}=\sup_{0=t_0<t_1<...<t_m= T}\big(\sum_{l:\, t_{l+1}-t_l\leq N^{-(1-\al)}}|\bar\bbW_N(t_i,t_{l+1})|^{p/2}\big)^{1/p}.
\]

Observe that
\[
\bbS^{ij}_N(s,t)=\bbS^{ij}_N(t)-\bbS^{ij}_N(s)-S^i_N(s)(S^j_N(t)-S^j_N(s)),
\]
\[
\bbW^{ij}_N(s,t)=\bbW^{ij}_N(t)-\bbW^{ij}_N(s)-W^i_N(s)(W^j_N(t)-W^j_N(s)).
\]
We will use this to estimate $J^{(1)}$ relying on
\begin{equation}\label{4.13}
\sup_{0\leq t\leq T}|S_N(t)-W_N(t)|=O(N^{-\del_4})\,\,\mbox{and}\,\,\sup_{0\leq t\leq T}|\bbS_N^{ij}(t)-\bbW_N^{ij}(t)|=O(N^{-\del_4})\,\,
\mbox{a.s.}
\end{equation}
which was already proved in Sections \ref{subsec3.3} and \ref{subsec4.2} for all $i,j$ and some $\del_4>0$ which does not depend on $N$. Since there
exist no more than $[TN^{1-\al}]$ disjoint intervals $[t_r,t_{r+1}]$ in $[0,T]$ with the length bigger than $N^{-(1-\al)}$ we obtain from (\ref{4.13})
that
\begin{equation}\label{4.14}
J_N^{(1)}\leq C_{\bbS\bbW}N^{(1-\al)p^{-1}-\frac 12\del_4}(1+\sup_{0\leq t\leq T}|S_N(t)|^{1/2}+\sup_{0\leq t\leq T}|W_N(t)|^{1/2})
\end{equation}
where $C_{\bbS\bbW}>0$ is an a.s. finite random variable which does not depend on $N$. Choosing $\al$ so that $2(1-\al)p^{-1}<\del_4$ and taking
into account that $\sup_{0\leq t\leq T}|W_N(t)|$ has all moments (with independent of $N$ bounds) we conclude
 by the Chebyshev inequality together with the Borel--Cantelli lemma that for any $\gam>0$,
\[
\sup_{0\leq t\leq T}|W_N(t)|=O(N^\gam)\quad\mbox{a.s.}
\]
which together with (\ref{3.22+}) yields also that
\[
\sup_{0\leq t\leq T}|S_N(t)|=O(N^\gam)\quad\mbox{a.s.}
\]
These together with (\ref{4.14}) imply that
\begin{equation}\label{4.15}
J_N^{(1)}=O(N^{1-\al-\frac 12p\del_4+\gam})\quad\mbox{a.s.}
\end{equation}

Next, in the same way as in Section 4.3 from \cite{Ki23} (see also Section 4.2 in \cite{FK}) we estmate
\begin{equation*}
E(J^{(2)}_N)^{4M}\leq C_{J,\be}N^{-2M(1-\al)(1-\frac 2p-\be)}\,\,\mbox{and}\,\,E(J^{(3)}_N)^{4M}\leq C_{J,\be}N^{-2M(1-\al)(1-\frac 2p-\be)}
\end{equation*}
where $\be>0$ can be chosen arbitrarily small and $C_{J,\be}>0$ does not depend on $N$. This together with the Chebyshev inequality
 and the Borel--Cantelli lemma yields, choosing as above $M$ large enough, that
\begin{equation}\label{4.16}
J^{(2)}_N=O(N^{-\frac 12(1-\al)(1-\frac 2p-\be)})\quad\mbox{and}\quad J^{(3)}_N=O(N^{-\frac 12(1-\al)(1-\frac 2p-\be)})\quad\mbox{a.s.}
\end{equation}.
Finally, Theorem \ref{thm2.1} follows for $\nu=2$ from (\ref{4.12}), (\ref{4.15}) and (\ref{4.16}) while its extension to $\nu>2$
will be obtained later on.

\section{Continuous time case}\label{sec5}\setcounter{equation}{0}
\subsection{Straightforward setup}\label{subsec5.1}
Set $\eta(k)=\int_{k}^{k+1}\xi(s)ds,\, k=0,1,...,[NT]-1$,
\[
Z_N(s,t)=N^{-1/2}\sum_{sN\leq k<[tN]}\eta(k),
\]
\[
\bbZ_N^{ij}(s,t)=N^{-1}\sum_{sN\leq l<k<[tN]}\eta_i(l)\eta_j(k)
\]
and, again, $Z_N(t)=Z_N(0,t)$, $\bbZ^{ij}_N(t)=\bbZ_N^{ij}(0,t)$. Observe that
\begin{eqnarray*}
&\|\eta(k)-E(\eta(k)|\cF_{k-l,k+l})\|_a\leq\int_k^{k+1}\|\xi(s)-E(\xi(s)|\cF_{k-l,k+l})\|_ads\\
&\leq 2\int_k^{k+1}\|\xi(s)-E(\xi(s)|\cF_{k-l+1,k+l-1})\|_ads\leq\be(a,l-1)
\end{eqnarray*}
where $\be$ defined by (\ref{2.15}) satisfies (\ref{2.4}) by the assumption of Theorem \ref{thm2.2}. Hence,
similarly to Lemma \ref{lem3.3} the limits
\begin{eqnarray*}
&\lim_{N\to\infty}N^{-1}\sum_{0\leq l<k<[tN]}E(\eta_i(l)\eta_j(k))=t\sum_{k=1}^\infty E(\eta_i(0)\eta_j(k))\\
&=\lim_{N\to\infty}N^{-1}\int_0^{tN}du\int_0^{[u]}E(\xi_i(v)\xi_j(u))dv=t(\int_0^\infty E(\xi_i(0)\xi_j(u))du-F_{ij})
\end{eqnarray*}
exist, where $F_{ij}=\int_0^1du\int_0^uE(\xi_i(v)\xi_j(u))dv$, and
\[
|N^{-1}\sum_{0\leq l<k<[tN]}E(\eta_i(l)\eta_j(k))-\sum_{1\leq k<[tN]} E(\eta_i(0)\eta_j(k))|=O(N^{-1}).
\]

Since the sequence $\eta(k),\, k\geq 0$ satisfies the conditions of Theorem \ref {thm2.1}, it follows
that the process $\xi$ can be redefined preserving its distributions on a sufficiently rich probability space which
contains also a $d$-dimensional Brownian motion $\cW$ with the covariance matrix $\vs$ so that the rescaled Brownian motion
$W_N(t)=N^{-1/2}\cW(Nt)$ satisfies
\begin{equation}\label{5.1}
\| Z_N-W_N\|_{p,[0,T]}=O(N^{-\ve})\,\,\mbox{and}\,\,\| \bbZ_N-\bbW_N^\eta\|_{\frac p2,[0,T]}=O(N^{-\ve})\quad\mbox{a.s.}
\end{equation}
where
\[
\bbW_N^\eta(s,t)=\int_s^tW_N(s,v)\otimes dW_N(v)+(t-s)\sum_{l=1}^\infty E(\eta(0)\otimes\eta(l))=\bbW_N(s,t)-(t-s)F,
\]
$ F=(F_{ij})$, $F_{ij}=\int_0^1du\int_0^uE(\xi_i(v)\xi_j(u))dv$, $\bbW_N=\bbW_N^{(2)}$ is defined in Section \ref{subsec2.1}, $p\in(2,3)$
and $\ve>0$ does not depend on $N$.
In fact, Theorem \ref{thm2.1} gives directly only that the sequence $\eta(k),\, k\geq 0$ above can be redefined so that $Z_N$ and $\bbZ_N$
constructed by this sequence satisfy (\ref{5.1}). But if $\eta(k),\, k\geq 0$ is the above sequence and $\tilde\eta(k),\,k\geq 0$
is the redefined sequence with $Z_N$ and $\bbZ_N$ constructed by it satisfy (\ref{5.1}) then we have two pairs of processes $(\xi,\eta)$ and $(\tilde\eta,\cW)$.
 Since the second marginal of the first pair has the same distribution as the first marginal of the second pair, we
can apply Lemma A1 of \cite{BP} which yields three processes $(\hat\xi,\hat\eta,\hat\cW)$ such that the joint distributions of
 $(\hat\xi,\hat\eta)$ and of $(\hat\eta,\hat\cW)$ are the same as of $(\xi,\eta)$ and of $(\tilde\eta,\cW)$, respectively. Hence,
 we have (\ref{5.1}) for $\hat\xi,\hat\eta,\hat\cW$ in place of $\xi,\eta,\cW$ and we can and will assume in what follows that (\ref{5.1})
 holds true for $\xi,\eta,\cW$ with $Z,\bbZ,W_N$ and $\bbW_N$ defined above.

 Thus, in order to derive Theorem \ref{thm2.2} for $\ell=1$ and $\ell=2$ in this setup it remains to estimate $\| Z_N-S_N\|_{p,[0,T]}$ and
 $\| \bbZ_N-\bbS_N+F\|_{p/2,[0,T]}$ where
 \begin{eqnarray*}
 &S_N(t)=\bbS_N^{(1)}(t)=N^{-1/2}\int_0^{tN}\xi(s)ds\,\,\,\mbox{and}\\
 &\bbS^{ij}_N(t)=\bbS^{ij}_N(t)=N^{-1}\int_0^{tN}\xi_j(u)du\int_0^u\xi_i(v)dv.
 \end{eqnarray*}
 Set $\bbF(s,t)=(t-s)F$ with the matrix $F$ defined above.
It was shown in Section 5 from \cite{Ki23} that
\begin{eqnarray}\label{5.2}
&E\| S_N-Z_N\|^{4M}_{p,[0,T]}\leq C_{SZp}(M)N^{-2M(1-\frac 2p)+1}\,\,\mbox{and}\\
& E\|\bbZ_N+\bbF-\bbS_N\|^{2M}_{\frac p2,[0,T]}\leq C_{SZp}(M)N^{-M\ka}\nonumber
\end{eqnarray}
for some $C_{SZp}(M)>0$ and $\ka>0$ which do not depend on $N$ provided $M$ is sufficiently large. Now by (\ref{5.2}) and the Chebyshev inequality,
\begin{eqnarray*}
&P\{\| S_N-Z_N\|^{4M}_{p,[0,T]}>N^{-\gam}\}\leq C_{SZp}(M)N^{-2M(1-\frac 2p-2\gam)+1}\,\,\mbox{and}\\
&P\{\|\bbZ_N+\bbF-\bbS_N\|^{2M}_{\frac p2,[0,T]}>N^{-\gam}\}\leq C_{SZp}(M)N^{-M(\ka-4\gam)}.
\end{eqnarray*}
Choosing $0<\gam<\min(\ka/4,\,\frac 12-\frac 1p)$, $M>\max(\ka-4\gam)^{-1},\,(1-\frac 2p-2\gam)$ and applying the Borel--Cantelli lemma we obtain that
\begin{equation*}
\| S_N-Z_N\|_{p,[0,T]}=O(N^{-\gam})\,\,\mbox{and}\,\,\| \bbS_N-\bbZ_N-\bbF\|_{\frac p2,[0,T]}=O(N^{-\gam})\quad\mbox{a.s.}
\end{equation*}
which together with (\ref{5.1}) yields Theorem \ref{thm2.2}.

 \subsection{Suspension setup}\label{subsec5.2}
Set again
\[
Z_N(s,t)=N^{-1/2}\sum_{sN\leq k<[tN]}\eta(k)\quad\mbox{and}
\]
\[
\bbZ_N^{ij}(s,t)=N^{-1}\sum_{sN\leq l<k<[tN]}\eta_i(l)\eta_j(k)
\]
with $\eta$ now defined in Section \ref{subsec2.3}. As before, we denote also $Z_N(t)=Z_N(0,t)$
and $\bbZ^{ij}_N(t)=\bbZ_N^{ij}(0,t)$. Since the sequence $\eta$ satisfies the conditions of Theorem \ref{thm2.1},
we argue similarly to Section \ref{subsec5.1} to conclude that the process $\xi$ can be redefined preserving its
distributions on a sufficiently rich probability space which contains also a $d$-dimensional Brownian motion
$\cW$ with the covariance matrix $\vs$ given by (\ref{2.19}) so that the rescaled Brownian motion $W_N(t)=
N^{-1/2}\cW(Nt)$ satisfies
\begin{equation}\label{5.3}
\| Z_N-W_N\|_{p,[0,T]}=O(N^{-\ve})\,\,\mbox{and}\,\,\| \bbZ_N-\bbW_N^\eta\|_{\frac p2,[0,T]}=O(N^{-\ve})\quad\mbox{a.s.}
\end{equation}
where
\begin{eqnarray*}
&\bbW_N^\eta(s,t)=\int_s^tW(s,v)\otimes dW_N^j(v)+(t-s)\sum_{l=1}^\infty E(\eta(0)\otimes\eta(l))\\
&=\bbW_N(s,t)-(t-s)F,\, F=(F_{ij}),\, F_{ij}=E\int_0^{\tau(\om)}\xi_j(s,\om)ds\int_0^s\xi_i(u,\om)du,
\end{eqnarray*}
$\bbW_N=\bbW_N^{(2)}$ is defined in Section \ref{subsec2.1}, $p\in(2,3)$, $\ve>0$ and $C_Z(M),\, C_\bbZ(M)>0$ do not depend on $N$.

Next, define $n(s)=n(s,\om)=0$ if $\tau(\om)>s$ and for $s\geq\tau(\om)$,
\[
n(s)=n(s,\om)=\max\{ k:\,\sum_{j=0}^{k-1}\tau\circ\vt^j(\om)\leq s\}.
\]
Now define
\begin{eqnarray*}
&U_N(s,t)=N^{-1/2}\sum_{n(s\bar\tau N)\leq k<n(t\bar\tau N)}\eta(k)\,\,\mbox{and}\\
&\bbU_N^{ij}(s,t)=N^{-1}\sum_{n(s\bar\tau N)\leq k<l<n(t\bar\tau N)}\eta_i(k)\eta_j(l)
\end{eqnarray*}
setting again $U_N(t)=U_N(0,t)$ and $\bbU^{ij}_N(t)=\bbU^{ij}(0,t)$. Comparing, first, $S_N(s,t)=\bbS_N^{(1)}(s,t)$
and $\bbS_N(s,t)-E\bbS_N(s,t)$ with $U_N(s,t)$ and $\bbU_N(s,t)$ and then with $Z_N(s,t)$ and $\bbZ_N(s,t)$,
respectively, it was shown in Section 6.2 of \cite{Ki23} that
\begin{equation}\label{5.4}
E\| S_N-Z_N\|^{4M}_{p,[0,T]}\leq C_{SZ}(M)N^{-4M\gam}
\end{equation}
for some $C_{SZ}(M),\gam>0$ which do not depend on $N$.  Set $\bbF(s,t)=(t-s)F$ with the matrix $F$ defined above.
In Section 6.3 of \cite{Ki23} it was shown also that
\begin{equation}\label{5.5}
E\|\bbS_N-\bbF-\bbZ_N\|^{2M}_{p/2,[0,T]}\leq \bbC_{\bbS\bbZ}(M)N^{-2M\gam}
\end{equation}
for some $\bbC_{\bbS\bbZ}(M),\gam>0$ which do not depend on $N$. Taking $M>1/\gam$ in (\ref{5.4}) and (\ref{5.5})
 and applying them together with the Chebyshev inequality to estimate
$P\{\| S_N-Z_N\|_{p,[0,T]}>N^{-\frac 12\gam}\}$ and $P\{\| \bbS_N-\bbF-\bbZ_N\|_{p/2,[0,T]}>N^{-\frac 12\gam}\}$
we derive via the Borel--Cantelli lemma that
\begin{equation}\label{5.6}
\| S_N-W_N\|_{p,[0,T]}=O(N^{-\frac 12\gam})\quad\mbox{and}\quad\| \bbS_N-\bbF-\bbZ_N\|_{p/2,[0,T]}=O(N^{-\frac 12\gam})\quad\mbox{a.s.}
\end{equation}
It follows from (\ref{5.3}) and (\ref{5.6}) that
\begin{equation}\label{5.7}
\| S_N-Z_N\|_{p,[0,T]}=O(N^{-\rho})\quad\mbox{and}\quad\| \bbS_N-\bbW_N\|_{p/2,[0,T]}=O(N^{-\rho})\quad\mbox{a.s.}
\end{equation}
where $\rho=\min(\ve,\gam/2)$.

\section{Multiple iterated sums and integrals}\label{sec6}\setcounter{equation}{0}

Let $\cD=\cD_{0T}=\{ 0=t_0<t_1<...<t_m=T\}$ be a finite partition of the interval $[0,T]$ and $\cD_{t_it_{i+1}}=\{ t_i=\tau_{i0}<\tau_{i1}<...
<\tau_{im_i}=t_{i+1}\}$ be partitions of $[t_i,t_{i+1}],\, i=0,1,...,m-1$ such that $N^{-\al}\leq\tau_{i,j+1}-\tau_{ij}<2N^{-\al}$ if $t_{i+1}
-t_i\geq 2N^{-\al}$ and if $t_{i+1}-t_i<2N^{-\al}$ then we take $m_i=1$ in which case $\cD_{t_it_{i+1}}$ consists of only one interval $[t_i,t_{i+1}]$.
Here $0<\al<1-\frac 2p$ is a small number. If $m_i>1$ then we derive by induction the following Chen type identities 
(see Theorem 2.1.2 in \cite{Lyo}, Example 3.1.1 in \cite{LQ} and Section 7.1 in \cite{Ki23}),
\begin{eqnarray*}
&\bbW_N^{(n)}(t_i,\tau_{i,r+1})=\bbW_N^{(n)}(t_i,\tau_{i,r})+\bbW_N^{(n)}(\tau_{ir},\tau_{i,r+1})\\
&+\sum_{k=1}^{n-1}\bbW^{(k)}(t_i,\tau_{ir})\otimes\bbW^{(n-k)}(\tau_{ir},\tau_{i,r+1})
\end{eqnarray*}
and
\[
\bbS_N^{(n)}(t_i,\tau_{i,r+1})=\bbS_N^{(n)}(t_i,\tau_{i,r})+\bbS_N^{(n)}(\tau_{ir},\tau_{i,r+1})+\sum_{k=1}^{n-1}
\bbS^{(k)}(t_i,\tau_{ir})\otimes\bbS^{(n-k)}(\tau_{ir},\tau_{i,r+1}).
\]
Summing these in $r=0,1,...,m_i-1$ we obtain,
\begin{equation*}
\bbW_N^{(n)}(t_i,t_{i+1})=\sum_{r=0}^{m_i-1}\bbW_N^{(n)}(\tau_{ir},\tau_{i,r+1})+\sum_{r=1}^{m_i-1}\sum_{k=1}^{n-1}
\bbW_N^{(k)}(t_i,\tau_{ir})\otimes\bbW_N^{(n-k)}(\tau_{ir},\tau_{i,r+1})
\end{equation*}
and
\begin{equation*}
\bbS_N^{(n)}(t_i,t_{i+1})=\sum_{r=0}^{m_i-1}\bbS_N^{(n)}(\tau_{ir},\tau_{i,r+1})+\sum_{r=1}^{m_i-1}\sum_{k=1}^{n-1}
\bbS_N^{(k)}(t_i,\tau_{ir})\otimes\bbS_N^{(n-k)}(\tau_{ir},\tau_{i,r+1}).
\end{equation*}
Then for $n\geq 3$,
\begin{equation}\label{6.1}
\sum_{i=0}^{m-1}\|\bbS_N^{(n)}(t_i,t_{i+1})-\bbW_N^{(n)}(t_i,t_{i+1})\|^{p/n}\leq\cI_\cD^{(1)}+\cI_\cD^{(2)}+\cI_\cD^{(3)}
\end{equation}
where
\[
\cI_\cD^{(1)}=\sum_{i=0}^{m-1}\|\sum_{j=0}^{m_i-1}\bbW_N^{(n)}(\tau_{ij},\tau_{i,j+1})\|^{p/n},\,\,
\cI_\cD^{(2)}=\sum_{i=0}^{m-1}\|\sum_{j=0}^{m_i-1}\bbS_N^{(n)}(\tau_{ij},\tau_{i,j+1})\|^{p/n}
\]
and
\begin{eqnarray*}
&\cI_\cD^{(3)}=\sum_{0\leq i<m,\, m_i>1}\|\sum_{j=1}^{m_i-1}\sum_{k=1}^{n-1}(\bbS_N^{(k)}(t_i,\tau_{ij})\otimes\bbS_N^{(n-k)}(\tau_{ij},\tau_{i,j+1})\\
&-\bbW_N^{(k)}(t_i,\tau_{ij})\otimes\bbW_N^{(n-k)}(\tau_{ij},\tau_{i,j+1})\|^{p/n}.
\end{eqnarray*}

In the same way as in Section 7.3 from \cite{Ki23} it follows that
\begin{equation}\label{6.2}
E\sup_\cD(\cI_\cD^{(1)})^{4M/p}\leq \bbC_{\bbW,\al}(M,T)N^{-2M\al(1-\al-\frac 2p)}\quad\mbox{and}
\end{equation}
\begin{equation}\label{6.3}
E\sup_\cD(\cI_\cD^{(2)})^{4M/p}\leq \bbC_{\bbS,\al}(M,T)N^{-2M\al(1-2\al-\frac 2p)},
\end{equation}
for some constants $\bbC_{\bbW,\al}(M,T)$ and $\bbC_{\bbS,\al}(M,T)$ which do not depend on $N$, and
\begin{eqnarray}\label{6.4}
&\cI_\cD^{(3)}\leq TN^\al\sum_{k=1}^{n-1}(\|\bbS_N^{(k)}\|^{p/n}_{p/k,[0,T]}\|\bbS_N^{(n-k)}-\bbW_N^{(n-k)}\|^{p/n}_{\frac p{n-k},[0,T]}\\
&+\|\bbS_N^{(k)}-\bbW_N^{(k)}\|^{p/n}_{p/k,[0,T]}\|\bbW_N^{(n-k)}\|^{p/n}_{\frac p{n-k},[0,T]}).\nonumber
\end{eqnarray}
Choosing $0<\al<1-\frac 2p$ and $M>(\al(1-\al-\frac 1p))^{-1}$ we estimate
\[
P\{ \sup_\cD(\cI_\cD^{(1)})>N^{-\frac \al 2(\frac p2(1-\al)-1)}\}\,\,\mbox{and}\,\, P\{ \sup_\cD(\cI_\cD^{(2)})>N^{-\frac \al 2(\frac p2(1-\al)-1)}\}
\]
by the Chebyshev inequality using (\ref{6.2}) and (\ref{6.3}) and then concluding by the Borel--Cantelli lemma that
\[
\sup_\cD(\cI_\cD^{(1)})+\sup_\cD(\cI_\cD^{(2)})=O(N^{-\frac \al 2(\frac p2(1-\al)-1)})\quad\mbox{a.s.}
\]
This together with (\ref{6.4}) yields that for $n\geq 3$,
\begin{eqnarray}\label{6.5}
&\|\bbS_N^{(n)}-\bbW_N^{(n)}\|_{p/n,[0,T]}^{p/n}\leq\bbC_\al^{(n)}N^{-\frac \al 2(\frac p2(1-\al)-1)}\\
&+ TN^\al\sum_{k=1}^{n-1}(\|\bbS_N^{(k)}\|^{p/n}_{p/k,[0,T]}\|\bbS_N^{(n-k)}-\bbW_N^{(n-k)}\|^{p/n}_{\frac p{n-k},[0,T]}\nonumber\\
&+\|\bbS_N^{(k)}-\bbW_N^{(k)}\|^{p/n}_{p/k,[0,T]}\|\bbW_N^{(n-k)}\|^{p/n}_{\frac p{n-k},[0,T]})\nonumber
\end{eqnarray}
where $\al>0$ was chosen above and $0<\bbC^{(n)}_{\al}<\infty$ is a random variable which does not depend on $N$.

 It was shown also in Section 7.3 of \cite{Ki23} that for $\nu=1,2,...,4M$,
\begin{equation}\label{6.6}
E\|\bbW_N^{(\nu)}\|^{4M/\nu}_{p/\nu,[0,T]}\leq T^{2M}\bbC_{\bbW}(M,\nu)<\infty
\end{equation}
and
\begin{equation}\label{6.7}
E\|\bbS_N^{(\nu)}\|^{4M/\nu}_{p/\nu,[0,T]}\leq T^{2M}\bbC_{\bbS}(M,\nu)<\infty.
\end{equation}
This together with the Chebyshev inequality and the Borel--Cantelli lemma yields that for any $\gam>\frac \nu{4M}$,
\begin{equation}\label{6.8}
\|\bbW_N^{(\nu)}\|_{p/\nu,[0,T]}=O(N^{\gam})\,\,\mbox{and}\,\,\|\bbS_N^{(\nu)}\|_{p/\nu,[0,T]}=O(N^{\gam})\,\,\mbox{a.s.}
\end{equation}

We already proved in Sections \ref{sec4} and \ref{sec5} that Theorems \ref{thm2.1}--\ref{thm2.3} hold true for $\nu=1,2$ with
some $\ve_1,\ve_2>0$ provided (\ref{2.4}) is satisfied for $M=M_0\geq 1$ large enough. Next, we proceed by induction. Suppose that
(\ref{2.4}) holds true for some $M\geq M_0$ and for $k=1,2,...,n-1,\, n\geq 3$ we have
\begin{equation}\label{6.9}
\|\bbS_N^{(k)}-\bbW_N^{(k)}\|_{p/k,[0,T]}=O(N^{-\ve_k})\quad\mbox{a.s.}
\end{equation}
where $\ve_1,\ve_2$ are as above and $\ve_k=\ve_{k-1}k^{-2}q$ for $k\geq 3$ where $q=\frac {p^2}{12}(\frac p2-1)$.
When $n\geq 3$  then by (\ref{6.5}), (\ref{6.8}) and (\ref{6.9}) with probability one,
\begin{equation}\label{6.10}
\|\bbS_N^{(n)}-\bbW_N^{(n)}\|_{p/n,[0,T]}^{p/n}= O(N^{-\frac {\al p}{2n}(\frac p2(1-\al)-1)})+O(N^{\frac {\al n}p+\gam-\ve_{n-1}}).
\end{equation}
Next, set $\gam=\frac 13\ve_{n-1}$, $\al=\min(\frac p{3n}\ve_{n-1},\,\frac 12-\frac 1p)$ and $\ve_n=\ve_{n-1}n^{-2}q$. Taking into account that
$\frac p2(1-\al)-1>\frac 12(\frac p2-1)$ with such choice of $\al$ (assuming without loss of generality that $\ve_2<1$), these choices
together with (\ref{6.10}) yield (\ref{2.10}) for $\nu=n\geq 3$ and $\ve_n=4\ve_2(n!)^{-2}q^{n-2}$. With the above choices of $\al$ and $\gam$,
in order to have (\ref{6.5}) and (\ref{6.8}) it suffices to assume that $M>\frac {9n}{\ve_{n-1}p}$, and so (\ref{2.10}) holds true for $\nu\leq\nu(M)$
provided, for instance, that $\frac 9{4\ve_2p}(\nu(M)!)^2q^{-\nu(M)}\leq M$, which completes the proofs.

\section{Law of iterated logarithm}\label{sec7}

As explained in Section \ref{subsec2.4}, in order to establish Theorem \ref{thm2.5} it suffices to prove (\ref{2.22}) and (\ref{2.23}).
\subsection{Proof of (\ref{2.22})}\label{subsec7.1}

In the discrete time case for any $t\in[0,T]$ we can write,
\begin{equation}\label{7.1}
\sup_{0\leq t\leq T}|\Sig^{(\nu)}(t\tau)-\Sig^{(\nu)}(t[\tau])|\leq\sup_{0\leq t\leq T}\sum_{t[\tau]\leq k<t\tau}|\xi(k)||\Sig^{(\nu-1)}(k)|\leq
 R_{\Sig}([\tau])
\end{equation}
where
\[
R_\Sig(N)=\max_{0\leq n\leq N}\sum_{nT\leq k<(n+1)T}|\xi(k)||\Sig^{(\nu-1)}(k)|.
\]
Then by the Cauchy--Schwarz inequality and the stationarity of the process $\xi$,
\begin{equation}\label{7.2}
E(R_\Sig(N))^{2M}\leq T^{2M-1}(E|\xi(0)|^{4M})^{1/2}\sum_{0\leq n\leq N}\sum_{nT\leq k\leq (n+1)T}(E|\Sig^{(\nu-1)}(k)|^{4M})^{1/2}.
\end{equation}

It follows from Section 7.1 of \cite{Ki23} that for any $k,l\geq 1$,
\begin{equation}\label{7.3}
E|\Sig^{(l)}(k)|^{4M}\leq C_{\Sig}(M)k^{2Ml}
\end{equation}
where $C_\Sig(M)>0$ does not depend on $k$ and $l$. Hence, by (\ref{7.2}),
\[
E(R_\Sig(N))^{2M}\leq C_R(M)N^{M(\nu-1)+1},
\]
and so
\begin{equation}\label{7.4}
P\{ R_\Sig(N)>N^{\frac 12(\nu-\frac 12)}\}\leq C_R(M)N^{-\frac M2+1}.
\end{equation}
Choosing $M>4$ we see that the right hand side of (\ref{7.4}) is a term of the converging series, and so by the Borel--Cantelli lemma
it follows that
\begin{equation}\label{7.5}
R_\Sig(N)=O(N^{\frac 12(\nu-\frac 12)})\quad\mbox{a.s.}
\end{equation}

Next, observe that by (7.20) from \cite{Ki23} for any $l\geq 1$ and $s>0$,
\begin{equation}\label{7.6}
E|\fW_\Gam^{(l)}(s)|^{2M}\leq C_\fW(M)s^{lM}.
\end{equation}
Now we write,
\begin{eqnarray}\label{7.7}
&\sup_{0\leq t\leq T}|\fW_\Gam^{(\nu)}(t\tau)-\fW^{(\nu)}(t[\tau])|\leq\sup_{0\leq t\leq T}|\int_{t\tau}^{t[\tau]}\fW_\Gam^{(\nu-1)}(s)\otimes d\cW(s)|\\
&+\sup_{0\leq t\leq T}|\int_{t\tau}^{t[\tau]}\fW_\Gam^{(\nu-2)}(s)\otimes\Gam ds|\leq R_\fW^{(1)}([\tau])+R^{(2)}_\fW([\tau])\nonumber
\end{eqnarray}
where
\[
R^{(1)}_\fW(k)=2\max_{0\leq n\leq k}\sup_{nT\leq t<(n+2)T}|\int_{nT}^{t}\fW_\Gam^{(\nu-1)}(s)\otimes d\cW(s)|
\]
\[
\mbox{and}\,\,\, R^{(2)}_\fW(k)=2\max_{0\leq n\leq k}\sup_{nT\leq t<(n+2)T}|\int_{nT}^{t}\fW_\Gam^{(\nu-2)}(s)\otimes\Gam ds|
\]
taking into account that for $nT\leq s<t<(n+2)T$,
\[
|\int_s^tadb|\leq |\int_{nT}^sadb|+|\int_{nT}^tadb|
\]
no matter what kind of integral we deal with. Now by (\ref{7.6}) and the standard martingale moment estimates for stochastic integrals
we obtain
\begin{eqnarray*}
&E|R^{(1)}_\fW(N)|^{2M}\leq 2^{2M}\sum_{0\leq n\leq N}E\sup_{nT\leq t<(n+2)T}|\int_{nT}^t\fW^{(\nu-1)}(s)\otimes d\cW(s)|^{2M}\\
&\leq C^{(1)}_{R_\fW}(M)\sum_{0\leq n\leq N}\int_{nT}^{(n+2)T}E|\fW^{(\nu-1)}(s)|^{2M}ds\leq \tilde C^{(1)}_{R_\fW}(M)N^{(\nu-1)M+1}
\end{eqnarray*}
for some $C^{(1)}_{R_\fW}(M),\,\tilde C^{(1)}_{R_\fW}(M)>0$ which do not depend on $N$. Estimating the probability
$P\{ R^{(1)}_\fW(N)>N^{\frac 12(\nu-\frac 12)}\}$ in the same way as in (\ref{7.4}) we conclude that
\begin{equation}\label{7.8}
R^{(1)}_\fW(N)=O(N^{\frac 12(\nu-\frac 12)})\quad\mbox{a.s.}
\end{equation}

Next, by (\ref{7.6}) and the Cauchy--Schwarz inequality,
\begin{eqnarray*}
&E|R_N^{(2)}(N)|^{2M}\leq 2^{2M}\sum_{0\leq n\leq N}E(\int_{nT}^{(n+2)T}|\fW_\Gam^{(\nu-2)}(s)\otimes\Gam|ds)^{2M}\\
&\leq 2^{4M+1}T^{2M-1}\sum_{0\leq n\leq N}\int_{nT}^{(n+2)T}E|\fW_\Gam^{(\nu-2)}(s)\otimes\Gam|^{2M})ds\leq C_{R_\fW}(M)N^{(\nu-2)M+1}.
\end{eqnarray*}
Proceeding as above we obtain by the Borel--Cantelli lemma that
\begin{equation}\label{7.9}
R^{(2)}_\fW(N)=O(N^{\frac 12(\nu-\frac 12)})\quad\mbox{a.s.}
\end{equation}
which together with (\ref{7.1}), (\ref{7.5}), (\ref{7.7}) and (\ref{7.8}) yields (\ref{2.22}).

\subsection{Proof of (\ref{2.23})}\label{subsec7.2}

First, we write
\begin{eqnarray}\label{7.10}
&\sup_{0\leq t\leq T}|\fW_\Gam^{(\nu)}(t\tau)-\fW^{(\nu)}(t\tau)|\leq\sup_{0\leq t\leq T}|\int_0^{t\tau}(\fW_\Gam^{(\nu-1)}(s)\\
&-\fW^{(\nu-1)}(s))d\cW(s)|+\sup_{0\leq t\leq T}|\int_0^{t\tau}(\fW_\Gam^{(\nu-2)}(s)\otimes\Gam ds|.
\nonumber\end{eqnarray}
We will estimate the right hand side of (\ref{7.10}) by induction. For $\nu=1$,
\begin{eqnarray*}
&\sup_{0\leq t\leq T}|\fW_\Gam^{(1)}(t\tau)-\fW^{(1)}(t\tau)|=0\quad\mbox{and}\\
&\sup_{0\leq t\leq T}|\fW_\Gam^{(1)}(t\tau)|=\sup_{0\leq t\leq T}|\fW^{(1)}(t\tau)|=O(\sqrt {\tau\ln\ln\tau})\quad\mbox{a.s.}
\end{eqnarray*}
by the standard (Strassen's) law of iterated logarithm. For $\nu=2$ we have
\begin{eqnarray*}
&\sup_{0\leq t\leq T}|\fW_\Gam^{(2)}(t\tau)-\fW^{(2)}(t\tau)|=T\tau|\Gam|=O(\tau)\quad\mbox{and}\\
&\sup_{0\leq t\leq T}|\fW_\Gam^{(2)}(t\tau)|\leq\sup_{0\leq t\leq T}|\fW^{(2)}(t\tau)|+T\tau|\Gam|=O(\tau\ln\ln\tau)\quad\mbox{a.s.}
\end{eqnarray*}
by the law of iterated logarithm for iterated stochastic integrals from \cite{Bal}. Suppose that
\begin{equation}\label{7.11}
\sup_{0\leq t\leq T}|\fW_\Gam^{(\nu)}(t\tau)-\fW^{(\nu)}(t\tau)|=O(\tau^{\nu/2}(\ln\ln\tau)^{\frac \nu 2-1})\quad\mbox{a.s.}
\end{equation}
\begin{equation}\label{7.12}
\mbox{and}\quad\sup_{0\leq t\leq T}|\fW_\Gam^{(\nu)}(t\tau)|=O((\tau\ln\ln\tau)^{\nu/2})\quad\mbox{a.s.}
\end{equation}
for $\nu=1,2,...,n-1,\, n\geq 3$ and prove (\ref{7.11}) and (\ref{7.12}) for $\nu=n$.

By the recurrence formula for $\fW_\Gam^{(n)}$ we have that
\begin{equation}\label{7.13}
|\fW_\Gam^{(n)}(t\tau)|\leq|\int_0^{t\tau}\fW_\Gam^{(n-1)}(s)\otimes d\cW(s)|+\int_0^{t\tau}|\fW_\Gam^{n-2}(s)||\Gam|ds.
\end{equation}
By (\ref{7.12}) for $\nu=n-2$,
\begin{equation}\label{7.14}
\sup_{0\leq t\leq T}\int_0^{t\tau}|\fW_\Gam^{(n-2)}(t\tau)||\Gam|ds=O(\tau^{n/2}(\ln\ln\tau)^{\frac n2-1})\quad\mbox{a.s.}
\end{equation}
Next, observe that by (\ref{7.12}) for $\nu=n-1$ we have also an estimate of the quadratic variation of the stochastic integral in
the right hand side of (\ref{7.13}) that has the form
\begin{equation}\label{7.15}
\langle\int_0^{t\tau}\fW_\Gam^{(n-1)}(s)\otimes d\cW(s)\rangle=O(\int_0^{T\tau}|\fW_\Gam^{(n-1)}(s)|^2ds)=O(\tau^n(\ln\ln\tau)^{n-1})\quad\mbox{a.s.}
\end{equation}
By the law of iterated logarithm for stochastic integrals (see, for instance, \cite{Wan}) we obtain from (\ref{7.15}) that
\[
\sup_{0\leq t\leq T}|\int_0^{t\tau}\fW_\Gam^{(n-1)}(s)\otimes d\cW(s)|=O((\tau\ln\ln\tau)^{n/2})\quad\mbox{a.s.}
\]
which together with (\ref{7.14}) proves (\ref{7.12}) for all $\nu\geq 1$.

Next, by (\ref{7.11}) for $\nu=n-1$,
\begin{eqnarray}\label{7.16}
&\langle\int_0^{t\tau}(\fW_\Gam^{(n-1)}(s)-\fW^{(n-1)}(s))\otimes d\cW(s)\rangle\\
&=O(\int_0^{T\tau}|\fW_\Gam^{(n-1)}(s)-\fW^{(n-1)}(s)|^2ds)=O(\tau^n(\ln\ln\tau)^{n-3})\quad\mbox{a.s.}\nonumber
\end{eqnarray}
This together with the law of iterated logarithm for stochastic integrals (see \cite{Wan} and references there) yields that
\begin{equation}\label{7.17}
\sup_{0\leq t\leq T}|\int_0^{t\tau}(\fW_\Gam^{(n-1)}(s)-\fW^{(n-1)}(s))\otimes d\cW(s)|=O(\tau^{n/2}(\ln\ln\tau)^{\frac n2-1})\quad\mbox{a.s.}
\end{equation}
Now (\ref{7.11}) for $\nu=n$ follows from (\ref{7.10}), (\ref{7.14} and (\ref{7.17}) completing the induction step. Thus, (\ref{7.11}) holds
true for all $\nu\geq 1$ implying (\ref{2.23}) which together with (\ref{2.22}), (\ref{2.24}) and Proposition \ref{prop2.4} yields Theorem \ref{thm2.5}.

\section{A.s. central limit theorem: proofs.}\label{sec8}

As in \cite{LP} denote by $BL=BL(\cH_T(\bbR^{d\otimes\nu}), \|\cdot\|_{BL})$ the class of functions $f:\cH_T(\bbR^{d\otimes\nu})\to\bbR$ with $\| f\|_{BL}=\| f\|_L+\| f\|_\infty<\infty$. Here $\| f\|_\infty$ is
the supremum norm and
\[
\| f\|_L=\sup\{\frac {| f(x)-f(y)|}{|\| x-y|\|_\infty}:\, x,y\in\cH_T(\bbR^{d\otimes\nu}),\, x\ne y\}
\]
where $|\| z|\|_\infty=\sup_{t\in[0,T]}\sum_{1\leq i_1,...,i_d\le d}|z^{i_1,...,i_\nu}(t)|$ when $z(t)\in\bbR^{d\otimes\nu},\, t\in[0,T]$.
It is easy to see that the space $\cH_T(\bbR^{d\otimes\nu})$ is separable, and so similarly to \cite{LP} we can rely on arguments from \cite{Du}, Theorem 11.3.3, $b\Rightarrow c$ to conclude that the assertion of Theorem \ref{thm2.8}
will follow if for all $f$ from a certain countable set $\Phi\subset\cH_T(\bbR^{d\otimes\nu})$ and $P$-almost all $\om$,
\begin{equation}\label{8.1}
\lim_{N\to\infty}(\log N)^{-1}\sum_{1\le n\le N}n^{-1}f(\bbS_n^{(\nu)}(\cdot,\om))=Ef(\bbW_1^{(\nu)}(\cdot)).
\end{equation}
Actually, we will prove (\ref{8.1}) for any $f\in\cH_T(\bbR^{d\otimes\nu})$.
Observe that the distribution of $\bbW_1^{(\nu)}$ does not depend on a particular choice of the Brownian motion 
(with the covariance matrix $\zeta$) $\cW=\bbW_1^{(1)}$ by which $\bbW_1^{(\nu)}$ is constructed, and so the right hand side of (\ref{8.1}) does not depend on a choice of $\cW$, as well.

By (\ref{2.8}) for each $f\in BL$ and $P$-almost all $\om$,
\begin{eqnarray*}
&|f(\bbS_n^{(\nu)}(\cdot,\om))-f(\bbW_n^{(\nu)}(\cdot,\om))|\\
&\leq\| f\|_L|\|\bbS_n^{(\nu)}(\cdot,\om))-\bbW_n^{(\nu)}(\cdot,\om)|\|_\infty=O(n^{-\ve_\nu}).
\end{eqnarray*}
Hence, (\ref{8.1}) will follow if we show that for $P$- almost all $\om$,
\begin{equation}\label{8.2}
\lim_{N\to\infty}(\log N)^{-1}\sum_{1\le n\le N}n^{-1}f(\bbW_n^{(\nu)}(\cdot,\om))=Ef(\bbW_1^{(\nu)}(\cdot))
\end{equation}
where $f\in BL$ is arbitrary.

In order to obtain (\ref{8.2}) it suffices to show that for any $f\in BL$,
\begin{equation}\label{8.3}
\lim_{N\to\infty}(\log N)^{-1}\sum_{1\le n\le N}n^{-1}Ef(\bbW_n^{(\nu)}(\cdot,\om))=Ef(\bbW_1^{(\nu)}(\cdot))
\end{equation}
and that
\begin{equation}\label{8.4}
\lim_{N\to\infty}(\log N)^{-1}\sum_{1\le n\le N}n^{-1}\up_n^{(\nu)}=0\,\,\, P-\mbox{a.s.}
\end{equation}
where $\up_n^{(\nu)}=\up_n^{(\nu)}(\om)=f(\bbW_n^{(\nu)}(\cdot,\om))-Ef(\bbW_n^{(\nu)}(\cdot))$.

We will see that for any integer $n\geq 1$ the processes $\bbW^{(\nu)}_n$ and $\bbW_1^{(\nu)}$ have the same distribution, and so, in particular, for any $n\geq 1$,
\begin{equation}\label{8.5}
Ef(\bbW_n^{(\nu)}(\cdot))=Ef(\bbW_1^{(\nu)}(\cdot)).
\end{equation}
This yields (\ref{8.3}) since
\[
\lim_{N\to\infty}(\log N)^{-1}\sum_{1\leq n\leq N}n^{-1}=1.
\]

In order to show that $\bbW^{(\nu)}_n$ and $\bbW_1^{(\nu)}$ have the same distribution we will consider for each integer $m\geq 1$ the partition of $[0,T]$ setting $t_k=T\frac kn,\, k=0,1,...,m$ and the map $\Psi_m^{(\nu)}$ of the space
$\cC([0,T],\,\bbR^m)$ of continuous paths $\gam(t)\in\bbR^m,\, t\in[0,T]$ defined inductively in the following way
$(\Psi^{(1)}_m\gam)(t)=\gam(t_k)$ if $t_k\leq t<t_{k+1}$,
\[
(\Psi^{(2)}_m\gam)(t)=\sum_{1\leq k\leq \frac {mt}T}\gam(t_{k-1})\otimes(\gam(t_k)-\gam(t_{k-1}))+T\Gam
\]
and for $\nu>2$,
\begin{eqnarray*}
&(\Psi^{(\nu)}_m\gam)(t)=\sum_{1\leq k\leq \frac {mt}T}\big((\Psi_m^{(\nu-1)}\gam)(t_{k-1})\otimes(\gam(t_k)-\gam(t_{k-1}))\\
&+(\Psi_m^{(\nu-2)}\gam)(t_{k-1})\otimes\Gam(t_k-t_{k-1})\big).
\end{eqnarray*}
Now observe that $\Psi_m^{(\nu)}W_n$ and $\Psi_m^{(\nu)}\cW$ have the same distribution since both $\cW$ and $W_n(t)=
n^{-1/2}\cW(nt),\, t\in[0,T]$ are Brownian motions (with the same covariation matrix), i.e. they have the same distribution.
It follows also by the definitions of the It\^ o and Riemann integrals that the distributions of $\Psi_m^{(\nu)}W_n$ and $\Psi_m^{(\nu)}\cW$ converge weakly as $m\to\infty$ to the distributions of $\bbW_n^{(\nu)}$ and $\bbW_1^{(\nu)}$, respectively, and so the latter processes have the same distribution.

It remains to establish (\ref{8.4}) for which, similarly to \cite{LP}, we will need the following result.
\begin{lemma}\label{lem8.1}
For each $f\in BL$ and $\up_k^{(\nu)}=f(\bbW_k^{(\nu)}(\cdot))-Ef(\bbW_k^{(\nu)}(\cdot))$,
\begin{equation}\label{8.6}
|E(\up_j^{(\nu)}\up_k^{(\nu)})|\leq C(j/k)^{1/2}\quad\,\mbox{for any}\quad 1\leq j\leq k
\end{equation}
where the constant $C>0$ depends only on $\| f\|_{BL}$. $d$, $T$ and $\nu$.
\end{lemma}
\begin{proof}
Set
\begin{eqnarray*}
R_{j,k}^{(\nu)}(t)=
\left\{
\begin{array}{ll}
0 & \mbox{if}\,\,\, 0\leq t\leq Tj/k\\
\bbW_k^{(\nu)}(Tj/k, t) & \mbox{if}\,\,\, Tj/k\leq t\leq T.
\end{array}
\right.
\end{eqnarray*}
It was shown in Section 7.1 of \cite{Ki23} (and is easy to see by induction) that iterated stochastic integrals $\bbW_k^{(\nu)}$ satisfy the Chen identities so that for $Tj/k\leq t\leq T$,
\begin{eqnarray}\label{8.7}
&\bbW_k^{(\nu)}(t)=\bbW_k^{(\nu)}(0,t)=\bbW_k^{(\nu)}(0,Tj/k)+\bbW_k^{(\nu)}(Tj/k,t)\\
&+\sum_{m=1}^{\nu-1}\bbW_k^{(m)}(0,Tj/k)\otimes\bbW_k^{(\nu-m)}(Tj/k,t).\nonumber
\end{eqnarray}

Set $\varrho_k^{(\nu)}=f(R^{(\nu)}_{j,k}(\cdot))-Ef(R^{(\nu)}_{j,k}(\cdot)).$ It is easy to see that $R_{j,k}^{(\nu)}$ is independent of $\up_j^{(\nu)}$, and so
\[
E(\up_j^{(\nu)}\up_k^{(\nu)})=E(\up_j^{(\nu)}(\up_k^{(\nu)}-\varrho_k^{(\nu)})).
\]
Hence, by the Cauchy--Schwarz inequality
\begin{eqnarray}\label{8.8}
&|E(\up_j^{(\nu)}\up_k^{(\nu)})|\leq (E(\up_j^{(\nu)})^2)^{1/2}(E(\up_k^{(\nu)}-\varrho_k^{(\nu)})^2)^{1/2}\\
&\leq 4\| f\|^2_{BL}(E|\|\bbW_k^{(\nu)}-R^{(\nu)}_{j,k}|\|^2_\infty)^{1/2}.\nonumber
\end{eqnarray}

To estimate the right hand side of (\ref{8.8}) we write
\begin{eqnarray}\label{8.9}
&\,\,\,\,\,\, E\|\bbW_k^{(\nu)}-R^{(\nu)}_{j,k}\|^2_\infty=E\sup_{t\in[0,T]}(\sum_{1\leq i_1,...,i_\nu\le d}|\bbW_k^{i_1,...,i_\nu}(t)-R_{j,k}^{i_1,...,i_\nu}(t)|)^2\\
&\leq 2d^\nu\sum_{1\le i_1,...,i_\nu\leq d}\big(E\sup_{0\leq t\le Tj/k}|\bbW_k^{i_1,...,i_\nu}(t)|^2\nonumber\\
&+E\sup_{Tj/k\le t\leq T}|\bbW_k^{i_1,...,i_\nu}(t)-\bbW_k^{i_1,...,i_\nu}(Tj/k,t)|^2\big).\nonumber
\end{eqnarray}
Now, by (\ref{8.7}) together with the Cauchy--Schwarz inequality and (\ref{6.6}) taken with appropriate $M$ and $T$
(i.e. replacing in (\ref{6.6}) $M$ by $m$ and $T$ by $Tj/k$ when estimating $E|\bbW_k^{i_1,...,i_m}(Tj/k)|^4$ below),
\begin{eqnarray}\label{8.10}
&\quad E\sup_{Tj/k\leq t\leq T}|\bbW_k^{i_1,...,i_\nu}(t)-\bbW_k^{i_1,...,i_\nu}(Tj/k,t)|^2\leq 2E|\bbW_k^{i_1,...,i_\nu}(Tj/k)|^2\\
&+2\nu\sum_{m=1}^{\nu-1}(E|\bbW_k^{i_1,...,i_m}(Tj/k)|^4)^{1/2}(E\sup_{Tj/k\leq t\leq T}|\bbW_k^{i_1,...,i_\nu}
(Tj/k,t)|^4)^{1/2}\nonumber\\
&\leq\tilde CTj/k\nonumber
\end{eqnarray}
Finally, combining (\ref{8.8})--(\ref{8.10}) we derive (\ref{8.6}).
Observe that we need here only supremum norm estimates which are simpler than the variational norm estimate (\ref{6.6}) as
they follow from standard moment estimates for stochastic integrals (see, for instance, \cite{Mao}).
\end{proof}   

We conclude the proof of (\ref{8.4}) in the same way as in \cite{LP}. Namely, set
\[
Z_l=\sum_{4^{l-1}\leq j<4^l}j^{-1}\up_j^{(\nu)},\quad l=1,2,....
\]
Then (\ref{8.4}) will follow if we show that
\begin{equation}\label{8.11}
\lim_{n\to\infty}\frac 1n\sum_{1\leq l\leq n}Z_l=0\quad P-\mbox{a.s.}
\end{equation}
Observe that by (\ref{8.6}) for all $1\leq l<m$,
\[
|E(Z_lZ_m)|=O(2^{l-m}). 
\]
Now (\ref{8.11}) follows from the strong law of large numbers for quasi-stationary sequences of random variables (see, for instance, Theorem 5.1.2 and its proof in \cite{Chu} or Corollary 2  in \cite{CLS}). This completes the proof of Theorem \ref{thm2.8}.


\begin{thebibliography}{Bow75}

\itemsep=\smallskipamount



\bibitem{Bal} P. Baldi, {\em Large deviations and functional iterated logarithm law for diffusion processes},
Prob. Theory Rel. Fields 71 (1986), 435--453.


\bibitem{Bow}
R. Bowen, {\em Equilibrium States and the Ergodic Theory of Anosov
Diffeomorphisms}, Lecture Notes in Math. 470, Springer--Verlag, Berlin, 1975.


\bibitem{Bra} R.C. Bradley, {\em Introduction to Strong Mixing Conditions,}
Kendrick Press, Heber City, 2007.




\bibitem{BP} I. Berkes and W. Philipp, {\em Approximation theorems for independent and
weakly dependent random vectors}, Annals Probab. 7 (1979), 29--54.



\bibitem{BBAK} P. Baldi, G. Ben Arous, G. Kerkyacharian, {\em Large deviations and the Strassen theorem in
H\" older norm}, Stoch. Proc. Appl. 42 (1992), 171--180.






\bibitem{BR} R. Bowen and D. Ruelle, {\em The ergodic theory of Axiom A
flows,} Invent. Math. {\bf 29} (1975), 181--202.




\bibitem{BV} I.S.Borisov and N.V. Volodko, {\em Limit theorems and exponential inequalities for canonical 
$U$- and $V$- statistics of dependent data}, IMS Collections: High Dimensional Probability, vol. 5 (2009),
108--130.


\bibitem{Bro} G. A. Brosamler, {\em An almost everywhere limit theorem}, Math. Proc. Camb. Phil. Soc 104
(1988), 561--574.

\bibitem{Chu} K.-L. Chung, {\em A Course in Probability}, 3d ed., Acad. Press,
San Diego, Ca., 2001.







\bibitem{CLS} S. Chebanyan, S. Levental, H. Salehi, {\em Strong law of large numbers under a general moment condition}, 
Elect. Comm. Probab. 10 (2005), 218--222.


\bibitem{CFKMZ} I. Chevyrev, P.K. Friz, A. Korepanov, I. Melbourne, H. Zhang,
{\em Deterministic homogenization under optimal moment assumptions for fast-slow systems.
Part 2}, Ann. l'Inst. H. Poincar\' e, Prob. Stat. 58 (2022), 1328--1350.



\bibitem{DK} M. Denker and G. Keller, {\em On $U$-statistics and v. Mises' statistics for weakly
 dependent processes}, Z. Wahrschein. Geb. 64 (1983), 505--522.
 
 
 
 \bibitem{DDP} H. Dehling, M. Denker and W. Philipp, {\em Invariance principles for von Mises and $U$-statistics},
 Z. Wahr. verw. Geb. 67 (1984), 139--167.
 




\bibitem{DP} H. Dehling and W. Philipp, {\em Empirical process techniques for dependent
data}, In: H.G. Dehling, T. Mikosch and M.S\o rensen (Eds.), {\em Empirical Process Techniques for
 Dependent Data}, p.p. 3--113, Birkh\" auser, Boston, 2002.


\bibitem{DET} J. Diehl, K. Ebrahimi-Fard and Nicolas Tapia, {\em Generalized iterated-sums signatures},
J. Algebra 632(2023), 801--824.


\bibitem{DR} J. Diehl and J. Reizenstein, {\em Invariants of multidimensional time series based
 on their iterated-integral signature}, Acta Appl. Math. 164 (2019), 83--122.
 
 
 \bibitem{Du} R.M. Dudley, Real Analysis and Probability, Wadsworth, Belmont, 1989.


\bibitem{DS}
N. Dunford and J.T. Schwartz, {\em Linear Operators}, Part I, Wiley, New York, 1958.











\bibitem{FH} P.K. Friz and M. Hairer, {\em A Course on Rough Paths}, Springer, Switzerland, 2020.


\bibitem{FK} P.K. Friz and Yu. Kifer, {\em Almost sure diffusion approximation in averaging via
rough paths theory},  Electron. J. Probab. 29 (2024), no.111, 1-56.


\bibitem{FZ} P.K. Friz and H. Zhang, {\em Differential equations driven by rough paths with jumps},
J. Diff. Equat. 264 (2018), 6226--6301.



\bibitem{Ga}
D.J.H. Garling, {\em Inequalities: a Journey into Linear Analysis}, Cambridge
Univ. Press, Cambridge (2007).





\bibitem{Hei}
L. Heinrich, {\em Mixing properties and central limit theorem for a class of
non-identical piecewise monotonic $C^2$-transformations},
Mathematische Nachricht. 181 (1996), 185--214.





 \bibitem{HL} B. Hambly and T. Lyons, {\em Uniqueness for the signature of a path of bounded variation
 and the reduced path group}, Ann. Math. 171 (2010), 109--167.








\bibitem{Ki22} Yu. Kifer, {\em Strong diffusion approximation in averaging with
dynamical systems fast motions}, Israel J. Math. 251 (2022), 595--634.


\bibitem{Ki23} Yu. Kifer, {\em Limit theorems for signatures}, arXiv: 2306.13376.



\bibitem{KY} S. Kanagawa and K. Yoshihara, {\em The almost sure invariance principles of degenerate
$U$-statistics of degree two for stationary variables}, Stoch. Proc. Appl. 49 (1994), 347--356.




 
\bibitem{LP} M.T. Lacey and W. Philipp, {\em A note on the almost sure central limit theorem}, Stat\& Probab. Letters 9
(1990), 201--205. 

 \bibitem{LQZ}  M. Ledoux, Z. Qian, T. Zhang, {\em  Large deviations and support theorem for diffusion processes via rough paths}, Stoch. Proc. Appl.
 102 (2) (2002) 265-283.

\bibitem{Lyo} T. Lyons, {\em Differential equations driven by rough signals}, Revista Mat. Iberoamericana 14.2 (1998): 215--310


\bibitem{LQ} T. Lyons and Z. Qian, {\em System Control and Rough Paths}, Clarendon Press, Oxford, 2002.


\bibitem{Mao} X. Mao, {\em Stochastic Differential Equations and Applications}, 2nd. ed.,
Woodhead, Oxford, 2010.


\bibitem{MN}
I. Melbourne and M. Nicol, {\em Almost sure invariance principle for nonuniformly
hyperbolic systems}, Commun. Math. Phys. 260 (2005), 131--146.




\bibitem{MP} D. Monrad and W. Philipp, {\em Nearby variables with nearby laws and a strong
approximation theorem for Hilbert space valued martingales}, Probab. Th. Rel. Fields 88
(1991), 381--404.





\bibitem{Wan} J-g. Wang, {\em A law of the iterated logarithm for stochastic integrals}, Stoch.
Proc. Appl. 47 (1993), 215--228.



\end{thebibliography}

\end{document}